\documentclass[a4paper,14pt]{article}
\pdfoutput=1
\usepackage{amsfonts}
\usepackage{amssymb}
\usepackage{amsmath}
\usepackage{amsthm}
\usepackage{verbatim}
\usepackage{booktabs}
\usepackage{tikz}
\usepackage{hyperref}

\newcommand{\real}{\mathbb{R}}

\newcommand{\nats}{\mathbb{N}}

\newcommand{\pars}[2]{\frac{\partial #1}{\partial #2}}

\newcommand{\lop}{\mathcal{L}_r}
\newcommand{\qop}{\mathcal{Q}_r}
\newcommand{\lops}{\mathcal{L}_\sigma}
\newcommand{\qops}{\mathcal{Q}_\sigma}
\newcommand{\lopp}{\hat{\mathcal{L}_\rho}}
\newcommand{\qopp}{\mathcal{Q}_\rho}
\newcommand{\fop}{\mathcal{F}_r}
\newcommand{\fops}{\mathcal{F}_\sigma}
\newcommand{\dopin}{\mathcal{D}_{\text{in}}}
\newcommand{\dfop}[2]{d\mathcal{F}_{#1}[#2]}
\newcommand{\fnrm}[1]{[#1]_r}
\newcommand{\fnrmrho}[1]{[#1]_\rho}

\newcommand{\Nb}{\mathcal{N}}
\newcommand{\Nbr}[1]{\mathcal{N}_{#1}}

\newcommand{\bry}{\mathfrak{B}}
\newcommand{\cry}{\mathfrak{C}}

\newcommand{\init}{\text{init}}
\newcommand{\formal}{\text{formal}}
\newcommand{\out}{\text{out}}
\newcommand{\para}{\text{para}}
\newcommand{\inner}{\text{in}}
\newcommand{\col}{\text{col}}

\newcommand{\defeq}{\mathrel{\mathop:}=}
\newcommand{\eqdef}{=\mathrel{\mathop:}}

\newtheorem{theorem}{Theorem}[section]
\newtheorem{lemma}[theorem]{Lemma}

\theoremstyle{remark}
\newtheorem*{remark}{Remark}

\DeclareMathOperator{\Rc}{Rc}
\DeclareMathOperator{\loc}{loc}
\DeclareMathOperator{\Rm}{Rm}
\DeclareMathOperator{\NP}{NP}
\DeclareMathOperator{\SP}{SP}

\title{Ricci Flow Emerging from Rotationally Symmetric Degenerate Neckpinches}
\author{Timothy Carson}
\begin{document}
\maketitle
\begin{abstract}
  In \cite{AIK} the authors created type-II Ricci flow neckpinch singularities.  In this paper we construct solutions to Ricci flow whose initial data is the singular metric resulting from these singularities.  We show in particular that the curvature decreases at the same rate at which it blew up.  This is the first example of Ricci flow starting from a type-II singularity.  
\end{abstract}
\tableofcontents
\section{Introduction}
The study of singularities is central to the study of Ricci flow.  In the absence of a weak formulation of Ricci flow, topological surgery is used to continue the flow through a singularity.  This was the technique proposed by Hamilton and completed by Perelman in dimension 3 to prove Thurston's geometrization conjecture. The flow through the singularity depends on the parameters of the surgery.  Perelman conjectured that, by sending the size of the surgery to zero, the surgically modified flows converge to a flow through the singularity, which we may call \emph{the} Ricci flow \cite{Perelman}.

In \cite{ACK}, Angenent, Caputo, and Knopf considered the singular metric on $S^{n+1}$ resulting from a rotationally symmetric regular neckpinch singularity.  They showed that there is a Ricci flow which emerges from that singular data, and it indeed can be constructed by the limit of small surgeries.  Their work furthermore tells us the precise asymptotics of the metric after the singularity.  In particular, the curvature $|\Rm|$ decreases at a rate slightly faster than the rate at which it blew up.  In a recent preprint \cite{Singular}, Kleiner and Lott have shown that in three dimensions Perelman's conjecture holds, although there is no proof that the resulting limit is independent of subsequence. 

This paper parallels \cite{ACK} for the case of degenerate neckpinches.  This is the first explicit example of Ricci flow through a type-II singularity.  We show that unlike the case of the regular neckpinch, the degenerate neckpinch heals at the same rate as it formed.  The asymptotics of $|\Rm|$ for both of these cases is shown in Table \ref{tab:curvatureasympts}.

\begin{table}[h]
\begin{tabular}{c | c | c}
  Asymptotics of $|\Rm|$ & Before singularity & After singularity \\
  \hline
  Type-I Neckpinch & $ \frac{1}{T-t} $ & $\frac{|\log(t-T)|}{t-T}$ \\
  Type-II Neckpinch & $ \frac{1}{(T-t)^{2-2/k}}  $ & $ \frac{1}{(t-T)^{2-2/k}}  $
\end{tabular}
\caption{The asymptotics of curvature before and after examples of neckpinches.  In the case of the degenerate neckpinch, $k \geq 3$ is an odd integer.}
\label{tab:curvatureasympts}
\end{table}

Recall that a finite time singularity of Ricci flow on $(0, T) \times M^{n}$ is called Type-I if $|\Rm| \sim (T-t)^{-1}$, and Type-II if the curvature blows up faster.  It appears to be the case that Type-II singularities only arise as degenerate cases.  In \cite{AIK}, Angenent, Isenberg, and Knopf constructed ``degenerate neckpinch'' metrics on $S^{n+1}$, $n \geq 2$, which develop a type-II singularity in finite time.  For any integer $k \geq 3$, there is a degenerate neckpinch singularity for which the curvature blows up as $|\Rm| \sim (T-t)^{2-2/k}$.  For $k$ even, the metric goes to $0$ on the entire manifold $S^{n+1}$.  On the other hand, for $k$ odd, only a set of partial $g(0)-$measure is destroyed.   We prove the existence of, and find asymptotics for, solutions to Ricci flow which emerge from the final metrics corresponding to odd $k$.  This paper fills in the lower-right hand corner of Table \ref{tab:curvatureasympts}.

Type-II singularities with other asymptotic blow-up rates have been observed.  In \cite{Needles}, Wu constructed non-compact examples which blow up at rate $(T-t)^{-(\lambda + 1)}$ for any $\lambda \geq 1$.

\subsection{Main Theorem}
\begin{figure}
    \centering
    \includegraphics[scale=1]{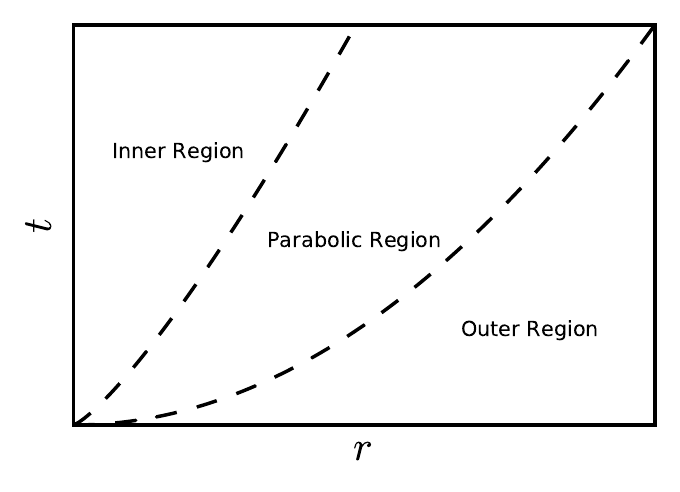}
    \caption{The three regions mentioned in the theorem.  The parabolic and inner regions shrink to nothing as $t \searrow 0$.}
\end{figure}
\begin{theorem}
  \label{maintheorem}
  Let $g_{\init}$ be a metric on $S^{n+1}-\{\NP\}$ satisfying the conditions in Section \ref{sec:initdata}.  Here $\NP$ is a point (the ``North Pole'') in $S^{n+1}$.  (In particular,  $g_{\init}$ could be the metric arising in the limit $t \nearrow T$ from the degenerate neckpinch for odd $k$, constructed in \cite{AIK}, with $b = 1-2/k$ below.)  There exists a smooth Ricci flow $g(t)$ on $S^{n+1} \times [0, t_*]$ emerging from $g_{\init}$.  Any such Ricci flow must be compact.  The flow may be constructed as the limit as $k \to \infty$ of $g_{\omega_k}(t)$, where each $g_{\omega_k}(t)$ is a solution to Ricci flow, and $g_{\omega_k}(0)$ is a modification of $g_{\init}$ in a small neighboorhood of $\NP$.

  In a neighboorhood of the North Pole the metric can be written as
  $$g(t) = \frac{(dr)^2}{v(r, t)} + r^2 g_{S^n}$$
  where $r$ is a coordinate in the neighboorhood, $v(r, t) \in [0,1]$, and $g_{S^n}$ is the round metric on $S^n$ with unit radius.

  The curvature tensor of the solution has norm $|\Rm| \sim  t^{-(1+b)} $ as $t \searrow 0$, and the solution $g(t)$ satisfies the following asymptotic profile.
  \begin{itemize}
  \item \textbf{Outer Region:} For $\rho_* \sqrt{t} < r < r_*$,
    $$v(r, t) = [1 + o(1)]\left(1 + 2(n-1)(1+b) \frac{t}{r^2} \right) r^{2b}.$$
  \item \textbf{Parabolic Region:} For $\sigma_* \sqrt{t^{1+b}} < r < \rho_* \sqrt{t}$,
    $$v(r, t) = [1 + o(1)]\left(1 + 2(n-1)\frac{t}{r^2}\right)^{1+b} r^{2b}.$$
  \item \textbf{Inner Region:} For $r < \sigma_* \sqrt{t^{1+b}}$,
    $$
    v(r, t) = [1 + o(1)]\bry
    \left(
      (2(n-1))^{-(b+1)/2}\frac{r}{\sqrt{t^{1+b}}}
    \right),
    $$
	where $\bry$ is such that the Bryant soliton is 
	$$g_{\text{Bryant}} = \frac{(dr)^2}{\bry (r)} + r^2 g_{S^n}.$$
  \end{itemize}
\end{theorem}

\subsection{Outline of Techniques}
The techniques we use follow closely those in \cite{ACK}.  All of the metrics we consider are rotationally symmetric.  Our initial metric is smooth except at one point, which we call the North Pole,  $\NP$.  Using the rotational symmetry, solving Ricci flow in a neighboorhood of the North Pole is equivalent to solving the quasilinear parabolic PDE
$$v_t = vv_{rr}-\frac{1}{2}v_r^2 + \frac{n-1-v}{r}v_r + \frac{2(n-1)}{r^2}v(1-v)$$
with boundary data $v(0, t) = 1$.  The initial singular metric corresponds to the initial data $v_{\init}(r) = (1 + o(1))r^{2b}$, where $b = 1-2/k \in (0, 1)$.  The fact that the initial metric is not smooth at the North Pole corresponds to the fact that $v_{\init}(0) \neq 1$.

In Section \ref{sec:coords} we describe the coordinate system we use throughout the paper, and the initial conditions which come out of the neckpinch constructed in \cite{AIK}. In Section \ref{sec:formal} we find a formal solution, and in Section \ref{sec:barriers} we construct barriers based on the formal solution.

In Section \ref{sec:conv} we prove our main theorem.  Lemma \ref{cptnesssubsoln} shows that any solution must be compact.  Results in Section \ref{section:curvaturebnds} show any solution within the barriers constructed in Section \ref{sec:barriers} satisfies curvature bounds. In Section \ref{section:regularizations} we construct modified initial metrics $g_\omega$ for small $\omega > 0$, which lie between the barriers, evaluated at time $\omega$.  Using the curvature bounds, we show a subsequence of the $g_{\omega}$ converge.  We show any solution which lies within our barriers satisfies curvature bounds, which allows us to extract a convergent subsequence of the modified metrics in Section \ref{section:convergencefinal}.

\subsection{Acknowledgements} 

I thank my advisor Dan Knopf for the suggestion of and significant help with the problem.
\section{Coordinates, Equations, and Initial Data}
\label{sec:coords}
\subsection{Basic Coordinates}
\label{sec:basiccoordinates}
In this paper we consider Ricci flow of $SO(n+1)$ invariant metrics on $S^{n+1}$.  For such a metric there are two well-defined poles $\{\NP, \SP\}$; by removing the two poles, we can identify $S^{n+1}$ with $(0, 1)\times S^{n}$, with the metric invariant under the action of $SO(n+1)$ on the $S^n$ factor. The metric can be written as
$$g = \varphi(x)^2(dx)^2 + \psi^2(x)g_{S^n},$$
where $g_{S^n}$ is the standard, round metric on $S^n$.
A coordinate which is more geometrically natural than $x$ is the arclength coordinate
$$s = \int_0^x\varphi(x)dx$$
in which the metric is written as
$$g = (ds)^2 + \psi(x)g_{S^n}.$$
\begin{figure}
  \centering
  \includegraphics[scale=1]{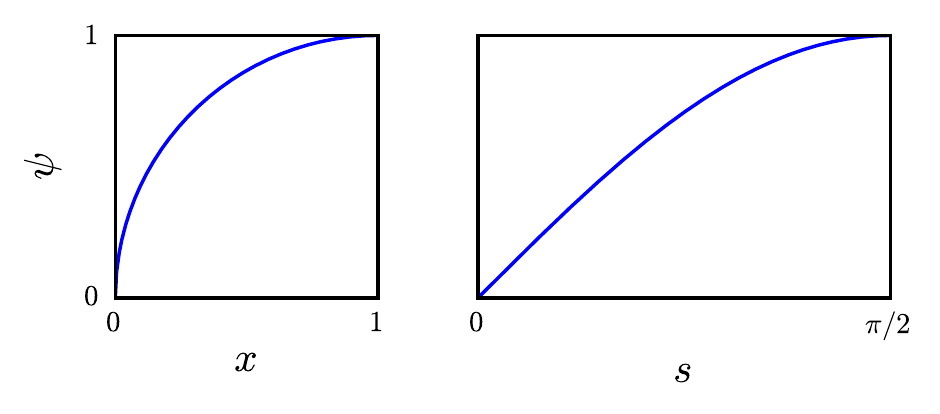}
  \caption{$\psi$ for half of the sphere $S^{n+1}$ with the standarded metric, in both the $\bar x$ and $s$ coordinates.  The point we call $\NP$ is the leftmost point of the sphere.}
\end{figure}

\begin{remark}
  If we want to see the manifold as the rotation of a graph over $\real$, it is necessary that $|\psi_s|\leq 1$.  (Otherwise, the height of the graph would have to increase faster than the arclength.)  In that case, we can define the coordinate
  $$\bar x = \int_0^s\sqrt{1-\psi_s^2}ds$$
  Then rotating the graph of $\psi$ in terms of $\bar x$ in $\real^{n+2}$ yields an embedding of the manifold.  
\end{remark}

There are two distinguished sectional curvatures:
$$K = \frac{-\psi_{ss}}{\psi}$$
which is the curvature of a plane containing $\frac{\partial}{\partial s}$, and
$$L = \frac{1-\psi_s^2}{\psi^2}$$
which is the curvature of a plane tangent to a spherical slice $\{x_0\}\times S^n$.   The Ricci curvature is given by
$$\Rc = nK(ds)^2 + (K + (n-1)L)\psi^2g_{S^n}.$$

In order for the addition of $\NP$ ($x = 0$) and $\SP$ ($x=1$) to yield a complete, compact, smooth manifold, it is necessary that $\psi$ is smooth in terms of $s$, and that 
\begin{equation}
  \psi_s\bigg|_{x=0} = 1, \qquad \psi_s\bigg|_{x=1}=-1
\end{equation}
to ensure that the sectional curvature $L$ does not blow up at the north or South Pole.

\begin{remark}
For smoothness it is also necessary that all even derivatives of $\psi$ (with respect to $s$) vanish at $x=0$ and $x=1$ \cite[Prop.~4.1]{AK}.  Instead of checking this directly we construct our solution as the limit of smooth flows, and we bound the derivatives of $\Rm$ for the smooth flows.
\end{remark}

Let $\Nb_{r_0}$ the connected neighboorhood of $\NP$ where $\psi(x)<r_0$.  As $\psi_s(0) = 1$, we may always choose $r_0$ small enough so that $\psi$ is increasing $\Nb_{r_0}$.  (This will also be true of our singular initial data, even though $\psi_{\init, s}(0) = 0$.)    Therefore we may use $\psi$ to replace $x$ as a coordinate in $\Nb_{r_0}$.  Set
$$r = \psi(x).$$
(So really $r$ and $\psi$ are the same thing, but we want to emphasize when we are using $r$ as a coordinate.)  The key function which describes the geometry in these coordinates is $\psi_s$; the metric can be written as
$$g = \frac{(dr)^2}{\psi_s^2} + r^2g_{S^n}.$$
We set
$$v(r) = \psi_s^2(r)$$
which has simpler evolution equations than $\psi_s$.  The sectional curvatures $K$ and $L$ become
$$K = \frac{-v_r}{2r} \qquad L = \frac{1-v}{r^2},$$
the Ricci curvature is
$$\Rc = \frac{nK}{v}(dr)^2 + (K + (n-1)L)r^2g_{S^n},$$
and the boundary condition necessary for smoothness at $\NP$ is $v(0) = 1$.

\begin{figure}
  \centering
  \includegraphics[scale=1]{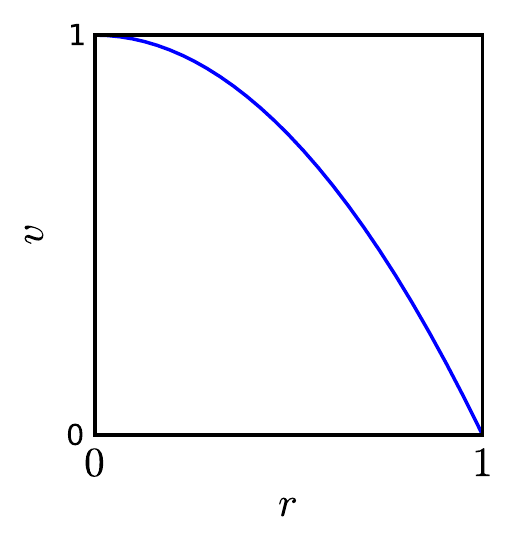}
  \caption{$v = \psi_s^2$ for half of the sphere $S^{n+1}$ with the standard metric.}
\end{figure}

The arclength between $r = r_0$ and $r=r_1$ is given by
$$\int_{r_0}^{r_1}\frac{1}{\sqrt{v}}dr$$
Thus a manifold with a complete, noncompact end, instead of the North Pole, would have $v$ going to zero at $r=0$ so fast that $v^{-1/2}$ is not integrable.

In the $s$ coordinate, Ricci flow reads
\begin{equation}\label{riccievopsi}
  \psi_t = \psi_{ss} - (n-1)\frac{1-\psi_s^2}{\psi}
\end{equation}
In the $r$ coordinate, Ricci flow reads
\begin{equation}\label{riccievo}
  v_t = \fop[v]
\end{equation}
where $\fop$ is the differential operator
$$\fop[v] = vv_{rr}-\frac{1}{2}v_r^2 + \frac{n-1-v}{r}v_r + \frac{2(n-1)}{r^2}v(1-v)$$
For convenience we write $a = 2(n-1)$\label{adef}. (One significance of $a$ is that the cylinder $\psi \equiv \sqrt{a(T-t)}$ is a solution of Ricci flow collapsing at time $T$.)
For the analysis, it's useful to see $\fop$ as
$$\fop[v] = \lop[v] + \qop[v]$$
where $\lop[v]$ and $\qop[v]$ are the linear and quadratic parts of $\fop[v]$, namely
\begin{equation}\label{opdef}
  \lop[v] = \frac{\frac{a}{2}rv_r + av}{r^2}
  \qquad
  \qop[v] =
  \frac{
    v(r^2v_{rr})
    - \frac{1}{2}(rv_r)^2
    - v(rv_r)
    - av^2
  }{r^2}
\end{equation}
We use $\qop[v_1, v_2]$ to mean the symmetric bilinear form such that $\qop[v] = \qop[v,v]$. We set $\fnrm{v}$ to be the pointwise norm
$$\fnrm{v} = |v| + r|v_r| + r^2|v_{rr}|.$$\label{fnrmdef}
The way we've written (\ref{opdef}) makes it clear that
\begin{equation}\label{opbnd}
  |\lop[v]| \leq C\frac{\fnrm{v}}{r^2}
  \qquad
  |\qop[v_1, v_2]| \leq C\frac{\fnrm{v_1}\fnrm{v_2}}{r^2}
\end{equation}
The bounds (\ref{opbnd}) will be the main way we control ``throwaway'' terms.

\subsection{Initial Data}\label{sec:initdata}
In this section we describe properties of the singular initial data $g_{\init}$.  In \cite{AIK}, a family of degenerate neckpinches was proven to exist, corresponding to integers $k \geq 3$.  The results we require are summarized in Theorem \ref{initdata} below.

\begin{theorem}\label{initdata}
  For any dimension $n \geq 2$ and odd integer $k \geq 3$, there exists an initial metric $g$ on $S^{n+1}$ such that
  \begin{itemize}
  \item The metric is rotationally invariant, as described above.  We will write $S^{n+1} = (-1, 1)\times S^n \cup \{\NP, \SP\}$.\footnote{whereas for the rest of this paper, $x$ lies in $(0, 1)$, although this is of small importance.}
  \item The metric $g$ satisfies the bound $|\psi \psi_{ss} - \psi^2_s + 1| < A < \infty$
  \item Under Ricci flow, the metric devolopes a singularity at time $T < \infty$.  On $(-1, 0) \times S^n$ the metric approaches $0$ as $t \nearrow T$.
  \item On $S^{n+1}-\{NP\}$ the metric stays smooth, and has a smooth limit as $t \nearrow T$.   For some $r_0$, the radius is increasing in the neighboorhood $\Nb_{r_0}$ of the north pole, so we can use $r = \psi$ as a coordinate.  Then $v(r, T) = \psi_s(x(r), T)^2$ satisfies
    $$v(r, T) = r^{2b} + O(r^{4b}), \qquad r \searrow 0$$
    where $b = 1-2/k$.
  \end{itemize}
\end{theorem}

Below we describe the properties we require for our initial data, all of which are satisfied by the limits as $t \nearrow T$ of the Ricci flows described in Theorem \ref{initdata}.  The initial data is described in terms of the radius function  $\psi_{\init}(s)$, where $s$ is the arclength from the singular tip $\NP$, and is defined for $s$ in $[0, l]$. 
\begin{enumerate}
\item The metric is defined and smooth on $S^{n+1}- \{ NP \}$.  The metric space completion is $S^{n+1}$. The metric is rotationally symmetric, as described in Section \ref{sec:basiccoordinates}.

\item For some $r_{\#}$, the radius, $\psi(s)$ is increasing on $\Nb_{r_{\#}}$, and the slope function $v_{\init}$ satisfies $v_{\init} = (1 + o(1))r^{2b}$ for some $b \in (0, 1)$.

For the neckpinches constructed in \cite{AIK}, the solution approaches $v(r) = (1 + o(1))r^{2(1-2/k)}$, where $k \geq 3$ is an odd integer.  We never use that $k$ is an integer, and it's cleaner to forget it, setting $b = 1-2/k$.
\label{bdef}
  
\item \label{vbnd} $|\psi_s(s)| < 1$ for $0 < s < l$.
  
  This inequality is preserved under Ricci flow \cite[Prop.~5.1]{AK}.  To understand this, recall
  $$\psi_t = \psi_{ss} - (n-1) (1-\psi_s^2)/\psi.$$
  So if $|\psi_s(s)| < 1$, then wherever $\psi_{ss}\leq 0$ we have $\psi_t \leq 0$.
  
\item \label{abnd} There is an $A > 0$ such that $|(1-v + \frac{r}{2}v_r)| = |r^2 (K-L)| < A$.

  Lemma 3.1 of \cite{AK} shows that the supremum over the manifold of $|r^2(K-L)|$ is nonincreasing.  This bound will be helpful because our control of the evolution is based on sub- and super-solutions for the slope function $v = \psi_s^2$.  Bounds on $v$ let us control the sectional curvatures $L = \frac{1-v}{r^2}$ which are tangent to the spherical cross sections, but not $K = -\frac{v_r}{2r}$. The bound on $(L-K)$ gets us past that problem. (This occurs in Lemma \ref{curvaturebndfixedtime}.)

\end{enumerate}

\section{Formal Matched Asymptotics}
\label{sec:formal}
In this section we construct a formal solution to the PDE $v_t = \fop[v]$.  Our initial metric is
$$v_{\init} = (1 + o_{r \to 0}(1))r^{2b}.$$
So for our formal solution we will have
$$v_{\formal}(r, 0) = v_0(r),$$
where $v_0(r) = r^{2b}$.
\subsection{Outer Region}
The outer region will be a time-dependent region where $r \gg \sqrt{t}$.  The exact definition will be determined when we construct barriers.  As a first approximation to $v_t = \fop[v]$, we consider
\begin{equation} \label{outerapprox}
  v \approx v_0 + t \fop[v_0].
\end{equation}
Calculate
\begin{align*}
  \fop[v_0] &=
  \frac{1}{r^2}
  \left[
    \left(
      \frac{a}{2}2b + a
    \right)r^{2b}
      + \left(2b(2b-1) - \frac{(2b)^2}{2} - 2b - a\right)r^{4b}
  \right]
\\
  &= \frac{1}{r^2}(1 + o(1))(a(1+b)r^{2b}).
\end{align*}
So take $v_{\out}$ to be the approximation
\begin{equation}\label{outerdef}
  v_{\out} = r^{2b}\left(1 + a(1+b)\frac{t}{r^2}\right).
\end{equation}
The approximation (\ref{outerapprox}) only makes sense as long as $v$ does not change very much, but if we look at the definition of $v_{\out}$ in (\ref{outerdef}) this is only true if $t \ll r^2$.  Set $\rho = \frac{r}{\sqrt{t}}$.  We expect $v_{\out}$ to be a good approximation where $\rho \gg 1$, and the outer region will be defined as $\rho \geq \rho_*$ for some $\rho_*$ to be determined.  We will have to find another approximation where $\rho \leq \rho_*$.

\subsection{Parabolic Region}
\label{sec:paraformal}
In the parabolic region, where $\rho \sim 1$, we will use the rescaled coordinates
$$\tau = \log(t), \qquad \rho = \frac{r}{\sqrt{t}} = re^{-\tau/2}.$$
% We let
% $$w(\rho(r, t), \tau(t)) = v(r, t), \qquad w_{\out}(\rho(r, t), \tau(t)) = v_{\out}(r, t), \qquad \text{ etc.}$$
% Note that this means that $v_{\out, t}$ is a time derivative keeping $r$ fixed, whereas $w_{\out, \tau}$ is the time derivative keeping $\rho$ fixed.
Write the approximation $v_{\out}$ in the parabolic coordinates:
$$v_{\out} = \rho^{2b}\left(1 + a(1+b)\rho^{-2}\right)e^{b\tau}.$$
For $\rho \sim 1$ and fixed $\tau$, $v_{\out}(\rho, \tau)$ is bounded.  We will try to find terms in the asymptotic expansion for $v$ in terms of $e^{b\tau}$:
\begin{equation}\label{paraexpansiontype}
  v \sim \sum_{j = 1}^\infty v^{(j)}(\rho)e^{jb\tau}
\end{equation}

If we write the Ricci evolution equations (\ref{riccievo}) in the parabolic coordinates we get
\begin{equation}\label{riccievopara}
  v_{\tau} = \lopp[v] + \qopp[v].
\end{equation}
The operator $\qopp[v]$ is the same as $\qop[v]$ defined in (\ref{opdef}) but with $r$'s and $r$ derivatives replaced with $\rho$'s and $\rho$ derivatives, and $\lopp[v]$ is
$$
\lopp[v] =
\frac{
  \frac{1}{2}a\rho v_{\rho} + av
}{
  \rho^2
}
+ \frac{1}{2}\rho v_\rho.
$$
By $v_{\tau}$ we will always mean the partial derivative of $v$ with respect to $\tau$, keeping $\rho$ fixed, whereas $v_t$ means the partial derivative with respect to $t$, keeping $r$ fixed.  Space derivatves are unambiguous, since fixing $\tau$ is the same as fixing $t$.  The extra term in $\lopp[v]$ compared to $\lop[v]$ comes from the time derivative
$$v_t
= v_\tau\tau_t   + v_\rho\rho_t
= v_\tau e^{-\tau} - \frac{1}{2}\rho e^{-\tau}$$

If we substitute the $e^{jb\tau}$ expansion (\ref{paraexpansiontype}) into the evolution equation (\ref{riccievopara}) and equate coefficients of $e^{b\tau}$ we get the ODE
$$bv^{(1)} - \lopp[v^{(1)}] = 0$$
for $v^{(1)}$.  We set $v_{\para}^{(1)}$ to be a solution to this ODE, which is given by
$$v_{\para}^{(1)}(\rho)
= K\rho^{2b}\left(1 + a\rho^{-2}\right)^{1+b}
= K\rho^{-2}\left(\rho^2 + a\right)^{1+b}$$
and set $v_{\para}$ to be the approximation
$$v_{\para}(\rho, \tau) = v_{\para}^{(1)}(\rho)e^{b\tau}.$$

We check that the parabolic approximation $v_{\para}$ can match the outer approximation $v_{\out}$ as $\rho \to \infty$ with fixed $\tau$:
$$
v_{\para}(\rho, \tau) = K\rho^{2b}(1 + (1+b)a\rho^{-2} + O(\rho^{-4}))e^{b\tau},
\qquad \rho \to \infty
$$
Thus we see if we choose $K = 1$ the approximations match to order $O(\rho^{2b-4})$.  Note that this error does go to zero as $\rho \to \infty$, because $b < 1$.

\begin{remark}
  In \cite{AIK}, which describes the shape of the neckpinch before the singularity, the corresponding region is called the ``intermediate'' region.  (The name ``parabolic'' is used for the region near the neck.)  The formal solution before the singularity is
  $$v_{\para}\big|_{t < 0} = \rho^2(a - \rho^{2})^{1+b}e^{-b\tau} = \rho^2(a - \rho^{2})^{1+b}e^{-b|\tau|}$$
  here $\tau$ is approaching $\infty$ as we approach the singularity (as $t \nearrow 0$).  This is similar to our formal solution
  $$v_{\para}\big|_{t > 0} = \rho^2(a + \rho^{2})^{1+b}e^{-b|\tau|}$$
  the difference -- the minus sign inside the parentheses -- causes the solution to become cylindrical ($v \to 0$) near $\rho = \sqrt{a}$. 
\end{remark}

The formal solution $v_{\para} = \rho^{-2}(\rho^2 + a )^{1+b}e^{b\tau}$ becomes large if $e^{b\tau} \gtrsim \rho^2$, i.e. if $t^{(b+1)/2} \gtrsim r$.  So we let
$$\sigma = rt^{-(b+1)/2}$$
be a rescaled space coordinate, and
$$\theta = t^{(b+1)/2}$$
be a rescaled time coordinate.  The left boundary of the parabolic region will be defined by
$$\sigma \gtrsim 1$$

\subsection{Inner Region}
\label{sec:innerformal}

We calculate the evolution of $v$ in the $(\sigma, \theta)$ (inner) coordinates. First calculate the partial derivatives
$$
v_t = \theta_t v_{\theta}  + \sigma_t v_{\sigma}
\qquad
v_r  = \sigma_r v_{\sigma} = \theta^{-1}v_\sigma
\qquad
v_{rr} =  \theta^{-2}v_{\sigma \sigma}.
$$
The partial derivative $v_{\theta}$ is always calculated keeping $\sigma$ fixed. Using $r \theta^{-1} = \sigma$,
$$
rv_r = \sigma v_\sigma, 
\qquad
r^2v_{rr} = \sigma^2 v_{\sigma \sigma}.
$$
Then by looking at the form of $\fop[v]$ in (\ref{opdef}) we find that $v_t = \fop[v]$ is equivalent to
\begin{equation}
  \theta_t v_\theta -\sigma \theta^{-1} \theta_t v_\sigma
  =
  \theta^{-2}\fops[W]
\end{equation}
(where $\fops$ is obtained from $\fop$ by replacing $r$ with $\sigma$) so
\begin{equation}\label{riccievotip}
  \theta\theta_t (\theta v_\theta - \sigma v_\sigma)
  =
  \fops[v].
\end{equation}

Notice that $\theta \theta_t = \frac{1}{2}(b+1)\theta^{2b/(b+1)}$.  If we put $v_{\para}$ in the $\sigma, \theta$ coordinates we see, for fixed $\sigma$, 
\begin{align}
  v_{\para}
  &= \sigma^{2b}(a\sigma^{-2} + t^b)^{1+b} \notag\\
  &= a^{1+b}\sigma^{-2} + (1+b)a^bt^b + O(t^{2b}) \notag\\
  &= a^{1+b}\sigma^{-2} + 2a^b\theta \theta_t + O((\theta \theta_t)^2) \label{paraInTip}
\end{align}
Inspired by this, and comforted by $\theta\theta_t \searrow 0$ as $t \searrow 0$, we look for the first couple of terms of a series expansion in powers of $\theta \theta_t$.  That is, we will make an approximation of the form
$$v_{\inner} = v_{\inner}^{(0)} + v_{\inner}^{(1)}\theta \theta_t$$
  
By now, the equations are exactly as in \cite{ACK}; the coordinates are different but the evolution equation (\ref{riccievotip}) and the boundary data (\ref{paraInTip}) are the same.  We summarize the reasoning here.

The constant-in-time term in the series expansion will solve
$$\fops[v_{\inner}^{(0)}] = 0$$
which has as its solutions
$$v_{\inner}^{(0)} = \bry(\kappa\sigma),$$
where $\bry$ is the function corresponding to the Bryant soliton, and $\kappa \in \real_{> 0}$ is arbitrary.   For further details of the Bryant solition in the current coordinate system, see appendix C of \cite{ACK}.  The important facts for us are
\begin{equation}\label{bryantinf}
  \bry(\sigma) = \frac{1}{\sigma^2} + o(\sigma^{-2}), \qquad \sigma \to \infty,
\end{equation}
\begin{equation}\label{bryantzero}
  \bry(\sigma) = 1 + b_2\sigma^2 + o(\sigma^{2}), \qquad \sigma \to 0,
\end{equation}
with $b_2 < 0$, and the smoothness condition $\bry(0) = 1$.
\begin{figure}
  \centering
  \includegraphics[scale=1]{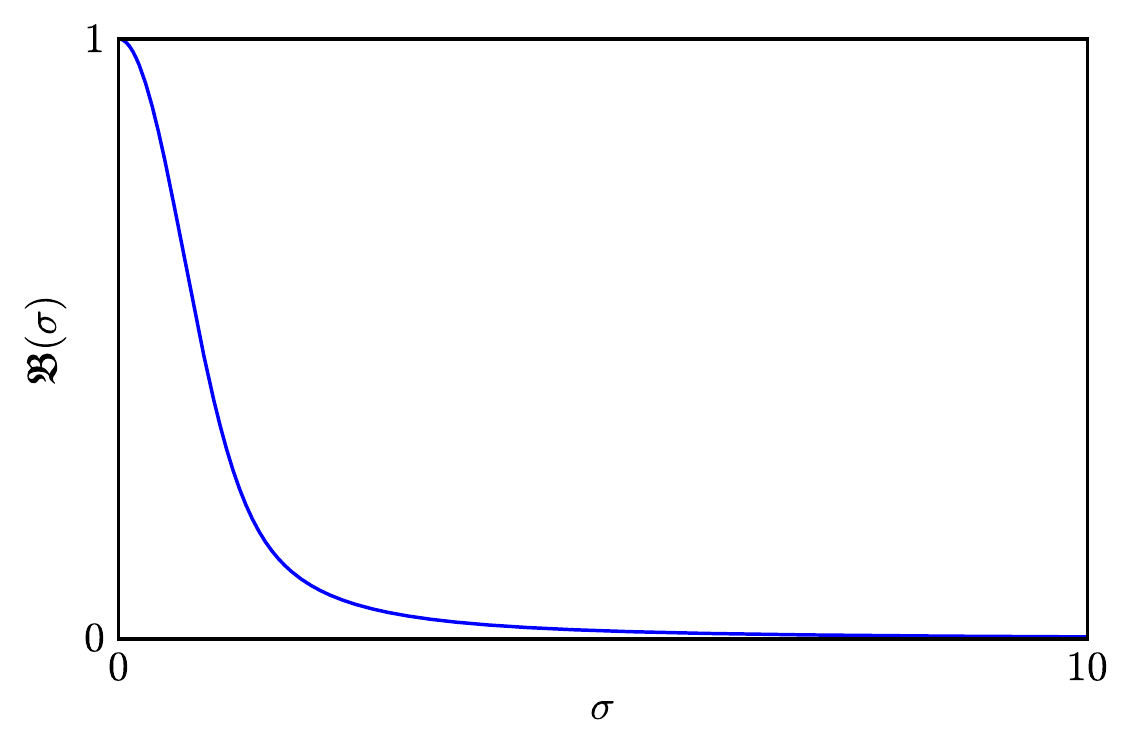}
  \caption{The Bryant soliton in the $r$ coordinate system.}
\end{figure}

Note that the first term in the series expansion, which we just computed, can only match the first term of (\ref{paraInTip}).  Now we search for the next term in the expansion.  Plugging $v = v_{\inner}^{(0)}(\sigma) + \theta \theta_t v_{\inner}^{(1)}(\sigma)$ into the left side of (\ref{riccievotip}) yields
\begin{align}
  \theta\theta_t(\theta v_\theta - \sigma v_\sigma)
  &= 
  \theta \theta_t (\theta (\theta\theta_t)_{\theta}v_{\inner}^{(1)} - \sigma (v^{(0)}_{\inner, \sigma} + \theta \theta_t v^{(1)}_{\inner, \sigma})) \notag\\
  & = -\theta\theta_t \sigma v^{(0)}_{\inner,\sigma} + o(\theta\theta_t),
  \qquad
  \theta \to 0
  \label{tipLHS}
\end{align}
and to the right side yields
\begin{align}
  \fops[v] &= \fops[v^{(0)}_{\inner} + \theta \theta_t v^{(1)}_{\inner}] \notag\\
  &= \fops[v^{(0)}_{\inner}] + \theta \theta_t \dfop{v^{(0)}_{\inner}}{v^{(1)}_{\inner}} + o(\theta\theta_t),
  \qquad \theta \to 0 \notag\\
  &= \theta \theta_t \dfop{v^{(0)}_{\inner}}{v^{(1)}_{\inner}} + o(\theta\theta_t),
  \qquad \theta \to 0.
  \label{tipRHS}
\end{align}
Here $\dfop{v^{(0)}_{\inner}}{v^{(1)}_{\inner}}$ is the differential of $\fops$ at $v^{(0)}_{\inner}$ in the direction $v^{(1)}_{\inner}$, that is
$$\frac{d}{d\epsilon}\fops[v^{(0)}_{\inner} + \epsilon v^{(1)}_{\inner}] \big|_{\epsilon = 0}$$
which is given by
$$\dfop{v^{(0)}_{\inner}}{v^{(1)}_{\inner}} = \lops[v^{(1)}_{\inner}] + 2 \qops[v^{(0)}_{\inner}, v^{(1)}_{\inner}].$$

So by equating (\ref{tipLHS}) and (\ref{tipRHS}) (the left- and right-hand sides of (\ref{riccievotip}))  we find that $v^{(1)}_{\inner}$ must satisfy
\begin{equation}
  \dfop{v^{(0)}_{\inner}}{v^{(1)}_{\inner}} = -\sigma v_{\inner 0,\sigma}
\end{equation}

The following lemma is \cite[Lemma~4]{ACK}\footnote{The statement and proof of the lemma in \cite{ACK} state $\cry(\sigma) = \frac{a}{2} + o(\sigma^{-2})$ as $\sigma \to \infty$, and later $\cry(\sigma) = \frac{2}{a} + o(\sigma^{-2})$ is used. Following the proof shows that the latter is correct, as we state below.}
\begin{lemma}\label{cryexistence}
  The equation
  \begin{equation}\label {cryexisteqn}
    \dfop{w_{0}}{w_{1}} = -\sigma w_{0,\sigma}, \qquad w_{0}(\sigma)= \bry(\sigma)
  \end{equation}
  has a strictly positive bounded solution $w_1 = \cry: [0, \infty) \to [0, \infty)$, with the asymptotics
  $$\cry(\sigma) = M\sigma^2  + o(\sigma^2), \qquad \sigma \to 0$$
  $$\cry(\sigma) = \frac{2}{a} + o(\sigma^{-2}), \qquad \sigma \to \infty$$
and any other solution which is bounded at 0 is given by $\cry(\sigma) + \lambda \phi(\sigma)$, where $\phi(\sigma) = -\sigma \bry'(\sigma)$ and $\lambda$ is arbitrary.
The solutions to (\ref{cryexisteqn}) with $w_{0} = \bry(\kappa\sigma)$ which are bounded at $0$  are given by
$$\frac{\cry(\kappa \sigma) + \lambda \phi(\kappa \sigma)}{\kappa^2}$$
\end{lemma}

In order to find the correct choice for $\lambda$ in an approximation, we would have to calculate more terms in the parabolic region, to match with the $\sigma^{-2}\theta \theta_t$ terms which are affected by $\lambda$.  We don't neeed such a good approximation for our purposes, so we just drop the $\lambda \phi (\kappa \sigma)$ term.

So, we take our formal solution to be
$$v_{\inner} = \bry(k\sigma) + \kappa^{-2}\cry(k\sigma)(\theta \theta_t).$$
We now find $k$ by matching with the parabolic solution.  Note
$$
v_{\inner} = \kappa^{-2}\sigma^{-2} + \kappa^{-2}\frac{2}{a}(\theta \theta_t) + o(\sigma^{-2}),
\qquad \sigma \to \infty.
$$
Comparing this with what we found in (\ref{paraInTip}), we find we must choose
$$\kappa = \kappa_0 \defeq a^{-(b+1)/2}$$
\label{k0def}

\section{Barriers}
\label{sec:barriers}
In this section we construct upper and lower barriers to the evolution equation (\ref{riccievo}).  By barriers, we mean properly ordered sub- and super-solutions to $(\pars{}{t} - \fop)$.  At time $0$, the barriers will surround the initial data $v_{\init}$.  In Section \ref{sec:conv} we will apply the maximum principle given by Lemma \ref{maxprinc} to the barriers we find in this section.

We first find barriers, based on our formal solutions, which are valid only in the outer, intermediate, and inner regions.  The barriers, and regions in which they are valid, are briefly summarized in Table \ref{tab:barriertable}.  Constants appear in the barriers, which must be chosen correctly.
\begin{table}[t]
  \begin{tabular}{l c  r}
    \toprule
    Outer
    &
    $(1 \pm \delta)v_0(r) + (1 \pm \epsilon) t \fop[(1 \pm \delta)v_0(r)]$
    &
    $\rho_* \sqrt{t} < r < r_*$
    \\
    Intermediate
    &
    $(1 \pm \gamma)v_{\para}^{(1)}e^{b\tau} \pm \frac{D}{\rho^4}e^{2b\tau}$
    &
    $\sigma_* e^{b\tau /2} < \rho < 3\rho_*$
    \\
    Inner
    &
    $\bry (\kappa_\pm \sigma) + (1 \mp \epsilon)\theta \theta_t \kappa_\pm^{-2} \cry(\kappa_\pm \sigma)$
    &
    $0 < \sigma < 3\sigma_*$
    \\
    \bottomrule
  \end{tabular}
  \caption{The barriers in each of the regions}
  \label{tab:barriertable}
\end{table}

The constants appearing in each of the barriers have restrictions on how they may be chosen, simply to ensure each barrier is a sub- or super-solution within its region of definition.  For example, in order to make $\epsilon$ smaller, we must increase $\rho_*$.

Furthermore, we will choose the constants to satisfy gluing conditions.  We explain these conditions for the case of the upper barrier; the lower barrier has similar conditions.  The upper barrier (across all regions) is defined as
$$
v^+=
\begin{cases}
  v_{\inner}^+ & \sigma \leq \sigma_*  \\
  \min (v_{\inner}^+, v_{\para}^+) & \sigma_* \leq \sigma \leq 3\sigma_* \\
  v_{\para}^+ & 3\sigma_*e^{b\tau/2} \leq \rho \leq \rho_* \\
  \min (v_{\para}^+, v_{\out}^+) & \rho_* \leq \rho \leq 3\rho_* \\
  v_{\out}^+ & 3\rho_*\sqrt{t} \leq r \leq r_*
\end{cases}
$$
For this to yield a piecewise smooth function we need
\begin{equation}\label{parainnnerglueconds}
v_{\inner}^+(\sigma_*) < v_{\para}^+(\sigma_*),
\qquad
v_{\para}^+(3\sigma_*) < v_{\inner}^+(3\sigma_*);
\end{equation}
\begin{equation}\label{paraoutglueconds}
v_{\para}^+(\rho_*) < v_{\out}^+(\rho_*),
\qquad
v_{\out}^+(3\rho_*) < v_{\para}^+(3\rho_*).
\end{equation}
These conditions are the gluing conditions for the supersolution. We will show in Sections  \ref{sec:outer_para_glue} and \ref{sec:para_inner_glue} that the constants may be chosen so that the barriers satisfy these conditions.

\begin{figure}
  \centering
  \scalebox{.7}{\begin{tikzpicture}
  [scale=.8,auto]
  \node (e) at (2,10)  {$\epsilon$};
  \node (d) at (10,10)  {$\delta$};
  \node (r) at (6,9) {$r_*$};
  \node (p) at (2,6) {$\rho_*$};
  \node (D) at (7,4)  {$D$};
  \node (g) at (13,4)  {$\gamma_\pm$};
  \node (s) at (2,2)  {$\sigma_*$};
  \node (k) at (13,0)  {$\kappa_\pm$};

  \draw [->] (r) to node {B1} (p);
  \draw [->] (e) to node {G1} (p);
  \draw [->] (e) to node {G1} (g);
  \draw [->] (d) to node {G1} (g);
  \draw [->] (d) to node {B0} (r);
  \draw [->] (p) to node {B2} (D);
  \draw [->] (p) to node {B2, G2} (s);
  \draw [->] (g) to node {G2} (k);
  \draw [->] (D) to node {G2} (k);
  \draw [->] (D) to node {B2, G2} (s);
  \draw [->] (s) to node {G2} (k);

  % \draw [->] (r) to node {\ref{outerbarriertheorem}} (p);
  % \draw [->] (e) to node {\ref{paraoutglue}} (p);
  % \draw [->] (e) to node {\ref{paraoutglue}} (g);
  % \draw [->] (d) to node {\ref{paraoutglue}} (g);
  % \draw [->] (p) to node {\ref{parabarriertheorem}} (D);
  % \draw [->] (p) to node {\ref{parabarriertheorem}, \ref{inoutglue}} (s);
  % \draw [->] (g) to node {\ref{inoutglue}} (k);
  % \draw [->] (D) to node {\ref{inoutglue}} (k);
  % \draw [->] (D) to node {\ref{parabarriertheorem}, \ref{inoutglue}} (s);
  % \draw [->] (s) to node {\ref{inoutglue}} (k);
  
\end{tikzpicture}}
  \caption{The directed acyclic graph showing the dependencies of the constants, e.g. $r_*$ depends on $\delta$. The edge labels reference how the dependencies come up: B0, B1, and B2 indicate the steps where we find barriers in the outer and parabolic regions, G1 and G2 indicate the two gluing steps.  All constants depend on $n$ and $b$, and the time at which the barriers are no longer valid ($t_*$) depends on everything and must be made smaller finitely many times.}
  \label{cnst_dag}
\end{figure}
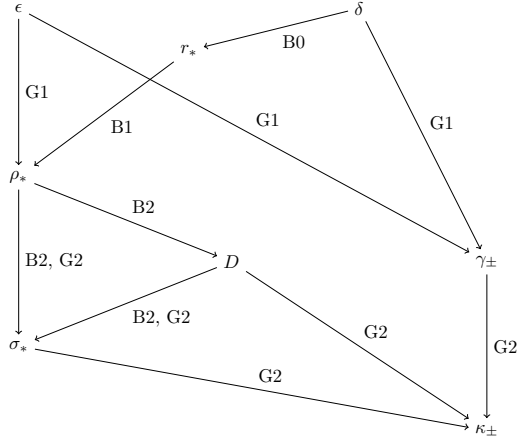
\subsection{Barriers in the Outer Region}

We state the outer barrier theorem here a little more generally, to highlight the inequalities which are vital to the argument.  Recall that $\fnrm{\cdot }$ was defined in (\ref{fnrmdef}) and gives a simple way to control $\lop$ and $\qop$:
\begin{equation}\label{opbnd2}
  r^2 \lop[v] \leq \fnrm{v}, \qquad r^2 \qop[v_1, v_2] \leq \fnrm{v_1}\fnrm{v_2}.
\end{equation}
\begin{theorem}\label{outerbarriertheorem}
  For all $\delta, \epsilon$ sufficiently small, the functions
  $$v_{\out}^\pm = (1 \pm \delta)v_0 + (1 \pm \epsilon)t\fop[(1\pm \delta)v_0]$$
  are properly ordered barriers in the region
  $$\{(r, t): \sqrt{t}\rho_* \leq r \leq r_*\}$$
  where $\rho_* = C/\sqrt{\epsilon}$, and $C$ depends on $v_0$ and $r_*$.
\end{theorem}

Recalling that our initial data $v_{\init}$ satisfies
$$v_{\init}(r) = (1 + o(1))v_0(r),$$
and taking a glance at the $(1 \pm \delta)$ factor in $v^{\pm}_{\out}$, we see that for any $\delta$ there is $r_*$ so that $v_{\out}^-(r,0) < v_{\init}(r) < v_{\out}^+(r,0)$ for $r \in [0, r_*]$, i.e. we have barriers for our initial data $v_{\init}$, not just our approximate initial data $v_0$.

\begin{proof}[Proof (of Theorem \ref{outerbarriertheorem}).]
  Set $w_0 = (1 \pm \delta)v_0$ and $\bar \epsilon = (1 \pm \epsilon)$.  The information we use about $v_0$ is the following. $v_0$ is continuous on $[0, r_*]$ and for $r \in [0, r_*]$ and $\delta$ sufficiently small,
  \begin{equation}\label{fopfnrmbnd}
    \fnrm{\fop[w_0]}
    \leq
    C|\fop[w_0]|
  \end{equation}
  and
  \begin{equation}\label{foppositive}
    \fop[w_0] > 0
  \end{equation}
To see these bounds, note that
$
r^2
\fnrm{
    \fop[u]
}
$
consists of terms of the form $r^j\frac{\partial ^ju}{\partial r^j}$, ($0 \leq j \leq 3$) and of the form $r^{j+l}\frac{\partial ^ju}{\partial r^j}\frac{\partial ^lu}{\partial r^l}$, ($0 \leq j+l \leq 4$).  Similarly, $r^2 \fnrm{\fop[u]}$ consists of terms of the form $r^j \frac{\partial ^j u}{\partial r^j}$, ($0 \leq j \leq 1$) and of the form $r^{j+l} \frac{\partial ^ju}{\partial r^j}\frac{\partial ^lu}{\partial r^l}$, ($0 \leq j + l \leq 2$).  Since $v_0= r^{2b}$ is just a power of $r$, both sides of (\ref{fopfnrmbnd}) are sums of powers of $r$, with the same smallest power, namely $r^{2b-2}$.  Furthermore, the coefficient on $r^{2b}$ in $\fop[r^{2b}]$ is $a(b+1) > 0$.  Thus (\ref{fopfnrmbnd}) and (\ref{foppositive}) hold for $v_0$, for small enough $r_*$, and for $w_0$ for small enough $\delta$.
  
Inequality (\ref{foppositive}) immediately shows that the barriers are properly ordered.  We only need to show that they are sub- and super-solutions.  

\begin{align}
(\partial_t -\fop)v_{out}^{\pm}
&= (1 \pm \epsilon)\fop[w_0]-\fop[w_0]
&&{}- \{ \fop[w_0 + (1 \pm \epsilon) t \fop[w_0]] - \fop[w_0]\}
\notag\\
&= \pm \epsilon \fop[w_0]
&&{}- \{ \fop[w_0 + (1 \pm \epsilon) t \fop[w_0]] - \fop[w_0]\}
\notag\\
&\eqdef \pm \epsilon \fop[w_0]
&&{}- H (r, t) \label{opbnd_outerbarthm}
\end{align}
The first term has the correct sign for $\delta$ sufficiently small, by inequality (\ref{foppositive}). We will bound the norm of the second term, which we have named $H(r, t)$.  Note that the second term is 0 for $t=0$, so at least for each $r$, the two terms together have the correct sign for small enough $t$.

Split $\fop$ into its linear and quadratic parts to see
$$
H(r, t)
= (1 \pm \epsilon) t \lop[\fop[w_0]]
+ 2(1 \pm \epsilon)  t \qop [w_0, \fop [w_0]]
+ (1 \pm \epsilon)^2 t^2 \qop[ \fop[w_0] ].
$$
We wish to bound this by a product of $\fop[w_0]$ with something we can make smaller than $\epsilon$, for use in (\ref{opbnd_outerbarthm}).  Using the bounds on $\lop$ and $\qop$ in (\ref{opbnd2}), then the upper bound in inequality (\ref{fopfnrmbnd}):
\begin{align*}
  |H(r,t)|
  &\leq C
  \left(
    t\frac{\fnrm{\fop[w_0]}}{r^2}
    + t \frac{\fnrm{w_0}\fnrm{\fop[w_0]}}{r^2}
    + t^2 \frac{\fnrm{\fop[w_0]}^2}{r^2}
  \right)\\
  &= C
  \left(
    1
    + \fnrm{w_0}
    + t\fnrm{\fop[w_0]}
  \right)
  \frac{t}{r^2}
  \fnrm{\fop[w_0]}\\
  &\leq C 
  \left(
    1 
    + \fnrm{w_0}
    + t\fop[w_0]
  \right)
  \frac{t}{r^2}
  \fop[w_0]
\end{align*}
Thus for $r \geq \rho_* \sqrt{t}$ we have
\begin{align*}
  |H(r,t)|
  &\leq C 
  \left(
    1 
    + \fnrm{w_0}
    + \frac{r^2}{\rho_*^2}\fop[w_0]
  \right)
  \frac{1}{\rho_*^2}
  \fop[w_0]
  \\
  &\leq C
  \left(
    1
    + \fnrm{w_0}
    + \frac{\fnrm{w_0}}{\rho_*^2}
    + \frac{\fnrm{w_0}^2}{\rho_*^2}
  \right)
  \frac{  1}{\rho_*^2}
  \fop[w_0]
\end{align*}
In the last line we again used the bound (\ref{opbnd2}) on $\lop$ and $\qop$.  Then, demanding $\rho_*>1$ and using continuity of $\fnrm{w_0}$ within $[0,r_*]$, we have
$$
|H(r, t)|
\leq
C\frac{1}{\rho_*^2}(1+2\fnrm{w_0} + \fnrm{w_0}^2)
\leq
C_{r_*}\frac{1}{\rho_*^2}\fop[w_0]
$$

Using this inequality in (\ref{opbnd_outerbarthm}),
$$
\partial_t v_{out}^{+} - \fop[v_{out}^+]
\geq \left( \epsilon - \frac{C_{r_*}}{\rho_*^2}\right)\fop[w_0]
$$
$$
\partial_t v_{out}^{-} - \fop[v_{out}^-]
\leq -\left(\epsilon - \frac{C_{r_*}}{\rho_*^2}\right)\fop[w_0]
$$
So for $\rho_* > \sqrt{\frac{C_{r_*}}{\epsilon}}$ the claim holds.
%\footnote{Due to considerations in theorem \ref{paraoutglue}, we also need to make sure $\rho_* > \sqrt{\frac{2a}{\epsilon}}$, which we can do by possibly increasing $C_{v_0, r_*}$.}
\end{proof}

\subsection{Barriers in the Parabolic Region}
Here we construct barriers based on the approximation in the parabolic region:
$$v \approx v_{\para}(\rho, \tau) = v^{(1)}_{\para}(\rho)e^{b\tau} = \rho^{-2}(a + \rho^2)^{1+b}e^{b\tau}$$
\begin{theorem}
\label{parabarriertheorem}
  For all $\rho_*$, there are $\sigma_*, D, \tau_*$ such that for all $\gamma < 1$
  $$v_{\para}^\pm = (1 \pm \gamma)v^{(1)}_{\para}e^{bt} \pm \frac{D}{\rho^4}e^{2bt}$$
  are barriers in the region
  $$
  \{(\rho, \tau): \sigma(\rho)>\sigma_*, \rho < 3\rho_*, 0 \leq \tau < \tau_*\}
  = \{(\rho, \tau): \sigma_* e^{b\tau/2} < \rho < 3\rho_*, 0 \leq \tau < \tau_*\}$$
\end{theorem}

\begin{proof}
  We are looking for sub and super solutions of the form
  $$v_{\para}^\pm = (1 \pm \gamma)v_1(\rho)e^{b \tau} \pm v_2(\rho)e^{2b\tau}$$
  where $v_1 = v^{(1)}_{\para} = \rho^{-2}(a + \rho^2)^{1+b}$.
    For this proof let $d = (1 \pm \gamma)$.  Calculate
  \begin{align}
    e^\tau(\partial_t v_{\para}^\pm - \fop[v_{\para}^\pm]) 
    &= \partial_\tau v_{\para}^\pm - \lopp[v_{\para}^\pm] - \qopp[v_{\para}^\pm]\nonumber\\
    &= d(bv_1 - \lopp[v_1])e^{b\tau}
     \pm (2bv_2-\lopp[v_2])e^{2b\tau}
    - d^2\qopp[v_1]e^{2b\tau}\nonumber\\
    &\phantom{{}=} \pm 2d\qopp[v_1, v_2]e^{3b\tau}
    - \qopp[v_2]e^{4b\tau} \nonumber\\
    &=\pm (2bv_2-\lopp[v_2])e^{2b\tau}
    - d^2\qopp[v_1]e^{2b\tau} \nonumber\\
    & \phantom{{}=} \pm 2d\qopp[v_1, v_2]e^{3b\tau}
    - \qopp[v_2]e^{4b\tau}\label{parabarrierexpr}
  \end{align}
  where for the last line we recall we defined $v_1$ as a solution to $bv_1-\lopp[v_1]=0$.  The plan is to make the first term have the correct sign, and bound the other terms.  

  \textbf{Bounding $\fnrmrho{v_1}:$}
  Calculate
  $$\frac{\partial_\rho v_1}{v_1} = \frac{1}{\rho}\frac{b\rho^2-a}{\rho^2+a}$$
  $$\frac{\partial_\rho^2v_1}{\partial_\rho v} =
  \frac{1}{\rho^2}
  \frac{
    (b^2-b)\rho^4 + (3-b)a\rho^2 + 2a^2
  }{
    (\rho^2 + a)(b\rho^2-a)
  }
  $$
  From which we see
  $$
  |\partial_\rho v_1| \leq \frac{C}{\rho}|v_1|
  \qquad
  |\partial_\rho^2 v_1| \leq \frac{C}{\rho^2}|v_1|
  $$
  so
  $$\fnrmrho{v_1} \leq C|v_1| = C\rho^{-2}(\rho^2+a)^{1+b}$$
  which gives
  $$\fnrmrho{v_1} \leq C\rho_*^{2+2b}\rho^{-2}$$
  provided we assume, say, that $\rho_*^2 \geq a$.

  \textbf{Causing the first $e^{2b\tau}$ term in (\ref{parabarrierexpr}) to swallow the second $e^{2b\tau}$ terms with a choice of $D$:}

  From the bound (\ref{opbnd}) on $\qopp$ with $\fnrmrho{\cdot}$:
  \begin{equation}\label{qoppbnd}
    |\qopp[v_1, v_1]| \leq \frac{C}{\rho^2}|v_1|^2 \leq  C\rho_*^{2+2b}\rho^{-6}
  \end{equation}

  The $\rho^{-6}$ inspired the choice $v_2 = D\rho^{-4}$.  Calculate
  \begin{equation}\label{otherstuffpos}
    (2bv_2-\lopp[v_2]) =
    D(a\rho^{-6} + 2(b+1)\rho^{-4} + a(2b+1)\rho^{4b-2}) \geq D\rho^{-6}
  \end{equation}
  % andn
  % $$2bv_2-\lopp[v_2] - \qopp[v_1, v_1] \geq (D-C)(\rho^{-6} + \rho^{4b-2})$$
  Comparing (\ref{qoppbnd}) and (\ref{otherstuffpos}) we see we can choose $D = 4C\rho_*^{2+2b}$ to obtain
  \begin{equation}\label{maincourse}
    (2bv_2-\lopp[v_2]-d^2\qopp[v_1, v_1])e^{2b\tau} \geq C\rho^{-6}e^{2b\tau}
  \end{equation}
  (taking the case of the super solution, for example)

  \textbf{Causing the $e^{2b\tau}$ terms in (\ref{parabarrierexpr}) to swallow the higher order terms by choosing $\sigma_*$}

  We calculate some bounds on the higher order terms.  Calculate
  $$\fnrmrho{v_2} \leq C_{\rho_*, D}\rho^{-4},$$
  which implies by the bound (\ref{opbnd})
  $$|\qopp[v_1, v_2]| \leq C_{\rho_*, D}\rho^{-8},$$
  $$|\qopp[v_2, v_2]| \leq C_{\rho_*, D}\rho^{-10}.$$
  Use these bounds on the higher order terms:
  \begin{equation}\label{scrapsbnd1}
    \qopp[v_1, v_2]e^{3b\tau}
    \leq C_{\rho_*, D}\rho^{-8}e^{3b\tau}
    = C_{\rho_*, D}\rho^{-2}e^{b\tau}\rho^{-6}e^{2b\tau}
  \end{equation}
  \begin{equation}\label{scrapsbnd2}
    \qopp[v_2, v_2]e^{4b\tau}
    \leq C_{\rho_*, D}\rho^{-10}e^{4b\tau}
    = C_{\rho_*, D}\rho^{-4}e^{2b\tau}\rho^{-6}e^{2b\tau}
  \end{equation}  
  Where on the right side, for both terms, we have peeled off a factor of $\rho^{-6}e^{2b\tau}$ to make comparison with (\ref{maincourse}) possible. The remaining factor is bounded in the region mentioned in the statement of the theorem.  That is, in the region
  $$
  e^{\tau/2}\sigma_* \leq \rho \qquad
  \left(
    \text{equivalently } \rho^{-2}e^{\tau} \leq \frac{1}{\sigma_*^2}
  \right)
  $$
  we have
  \begin{equation}\label{scrapsbnd1reg}
    \qopp[v_1, v_2]e^{3b\tau} \leq \frac{C_{\rho_*, D}}{\sigma_*^2}\rho^{-6}e^{2b\tau}
  \end{equation} 
  \begin{equation}\label{scrapsbnd2reg}
    \qopp[v_2, v_2]e^{4b\tau} \leq \frac{C_{\rho_*, D}}{\sigma_*^4}\rho^{-6}e^{2b\tau}
  \end{equation}
  
  So comparing (\ref{maincourse}) with (\ref{scrapsbnd1reg}, \ref{scrapsbnd2reg}) we see we can choose $\sigma_* = \sigma_*(\rho_*, D, \gamma)$ so that
  \begin{align*}
    e^\tau(\partial_t v_{\para}^+ - \fop[v_{\para}^+]) 
    &=
    (2bv_2-\lopp[v_2]- d^2\qopp[v_1])e^{2b\tau} \\
    &\phantom{{}=}+ 2d\qopp[v_1, v_2]e^{3b\tau}
    - \qopp[v_2]e^{4b\tau} \\
    &\geq C_{\rho_*, D}\rho^{-6}e^{2b\tau}
  \end{align*}
  and the corresponding inequality also holds for the subsolution.
\end{proof}

\subsection{Barriers in the Inner Region}
The following theorem gives sub- and super-solutions to be used in the inner region.  These are identical to \cite[Theorem~6]{ACK}.  As mentioned in the formal derivation of the inner solution (Section \ref{sec:innerformal}), our equations here are the same as the equations in \cite{ACK}; only the coordinates are different.  Recall that in the inner region we have the approximation
$$
v
\approx v^{(0)}_{\inner}(\sigma) + v^{(1)}_{\inner}(\sigma)(\theta \theta_t)
= \bry (\kappa_0 \sigma) + \kappa^{-2}\cry (\kappa_0 \sigma)
$$

\begin{theorem}
  For any $\sigma_*$ and $\epsilon$ there exists $t_* = t_*(\sigma_*, \epsilon)$ such that for any $\kappa \in [\frac{\kappa_0}{2}, 2\kappa_0]$ the functions
  $$v_{\inner}^\pm(\sigma, t) = \bry(\kappa\sigma) + (1 \mp \epsilon)\kappa^{-2}\cry(\kappa\sigma)\theta \theta_t$$
  are sub- and super-solutions in the region
  $$
  \{(\sigma, \theta): 0 \leq \sigma \leq 3\sigma_*t^{(b+1)/2}, \theta \leq t_*\}
  $$
\end{theorem}

We include the proof here.  The upper and lower barriers will have to have different choices of $\kappa$, because the $(1 \mp \epsilon)$ factor in the second term causes $v^\pm$ to not be properly ordered for fixed $\kappa$.

\begin{proof}
  For this proof let
  $$v_0(\sigma) = \bry(\kappa\sigma), \qquad  v_1(\sigma) =\cry(\kappa\sigma)/\kappa^2.$$
  We deal with the case of the subsolution, given by
  $$
  v_{\inner}^- = v_0(\sigma) + (1+\epsilon)v_1(\sigma) \theta \theta_t.
  $$
  Copying equation (\ref{riccievotip}), Ricci flow is given by $\dopin[v] = 0$, where $\dopin[v]$ is
  $$\dopin[v] = \theta\theta_t (\theta v_\theta - \sigma v_\sigma) - \fops[v]$$
  
  We calculate
  \begin{align}
    \dopin[v_{\inner}^-]
    &= (1+\epsilon)(\theta \theta_t) \theta (\theta \theta_t)_t v_1
    - (\theta \theta_t) \sigma v_{0, \sigma}
    - (\theta \theta_t)^2 \sigma v_{1,\sigma} \notag\\
    &\phantom{{}=}-\fops[v_0 + (1+\epsilon)\theta \theta_t v_1] \notag\\
    &= (1+\epsilon)(\theta \theta_t) \theta (\theta \theta_t)_t v_1
    - (\theta \theta_t) \sigma v_{0, \sigma}
    - (\theta \theta_t)^2 \sigma v_{1,\sigma}\notag\\
    &\phantom{{}=}- \fops[v_0] - (1+\epsilon)(\theta \theta_t)\dfop{v_0}{v_1} - (1+\epsilon)^2(\theta \theta_t)^2 \qop[v_1] \label{dopincalc}
  \end{align}
  Recall that $v_0$ and $v_1$ were chosen in Section \ref{sec:innerformal} so that the constant-in-time term and the $(\theta \theta_t)$ term of
  $$\dopin[v_0 + v_1 (\theta \theta_t)]$$
  are zero, i.e. so that
  $$\fops[v_0] = 0, \qquad  \dfop{v_0}{v_1} = -\sigma v_{0, \sigma} .$$
  Thus we cancel these terms in (\ref{dopincalc}) and arrive at
  \begin{equation}\label{dopincalc2}
    \dopin[v_{\inner}^-] =
    \epsilon  \sigma v_{0, \sigma} (\theta \theta_t)
    +  ((1+\epsilon) \theta v_1 - \sigma v_{1, \sigma} - (1+\epsilon)^2 \qops[v_1])(\theta \theta_t)^2
  \end{equation}
  Since $v_0$ is strictly decreasing, the first term, $\epsilon \sigma v_{0, \sigma}$,  has the correct sign.  We want to show that the first term dominates the rest, for small enough $t$. Since we are dealing with bounded $\sigma$, and the bound on $t$ may depend on the bound of $\sigma$, the only difficulty is at $\sigma = 0$; there the first term goes to 0.
  
  From the power expansion of $v_0$ at 0:
  $$v_0(\sigma) = 1 + b_2 \kappa^2\sigma^2 + o(\sigma^2), \qquad \text{for some } b_2 < 0,$$
  and since $v_0(\sigma)$ is strictly decreasing, we see that for $\sigma$ in $[0, \sigma_*]$ there is $C_{\sigma_*}>0$ such that
  $$ \sigma v_{0, \sigma} \leq - C_{\sigma_*}\sigma^2,
  \qquad \sigma \in [0, \sigma_*].
  $$
  (Since $\kappa$ is universally restricted to be in $[\frac{\kappa_0}{2}, 2\kappa_0]$ the onstant does not depend on $\kappa$.)
  Additionally, we can bound the spatial parts of the second term from above, with a similar bound.  Looking at the expansion of $\cry (\sigma)$ in Lemma \ref{cryexistence} we see that there is $C_{\sigma_*}'>0$ such that
  $$
  |v_1| + |\sigma v_{1, \sigma}| + |\qops[v_1]| \leq C_{\sigma_*}'\sigma^2,
  \qquad \sigma \in [0, \sigma_*]
  $$

  So applying these bounds to (\ref{dopincalc2}) we have
  $$
  \dopin[v_{\inner}^-]
  \leq
  -\epsilon C_{\sigma_*} (\theta \theta_t)
  +C_{\sigma_*}' (\theta \theta_t)^2
  $$
  from which we see that $v_{\inner}^-$ is a subsolution in the region desired if we choose $t_* = t_*(\sigma_*, \epsilon)$ small enough.
\end{proof}
\subsection{Gluing the Outer and Parabolic Barriers}
\label{sec:outer_para_glue}
In this section, we show that the gluing conditions for the outer and parabolic barriers, i.e. (\ref{paraoutglueconds}) and the corresponding inequalities for the subsolution, hold for certain choices of the constants in the barriers.  We outline the argument, taking the case of the supersolution as an example.  We show that we can choose $\rho_*, \gamma$ depending on $\epsilon, \delta$ so that
\begin{itemize}
\item $\lim_{\tau \to -\infty}v_{\out}^+(2\rho_*, \tau)/v_{\para}^+(2\rho_*, \tau) = 1$
\item The limit $\lim_{\tau \to -\infty}v_{\out}^+(\rho, \tau)/v_{\para}^+(\rho, \tau)$ is decreasing in $\rho$, for $\rho \in [\rho_*, 3\rho_*]$.
\end{itemize}
These facts show (\ref{paraoutglueconds}) holds in the limit $\tau \to -\infty$, and we choose $\tau_*$ small enough so the inequalities still hold for $\tau < \tau_*$.

\begin{theorem}\label{paraoutglue}
  With $v_{\out}^\pm$ and $v_{\para}^\pm$ defined as in Theorems \ref{outerbarriertheorem} and \ref{parabarriertheorem}, for any $\epsilon, \delta$ we can choose $\rho_*, \tau_*, \gamma_{\pm}$ so that $v_{\out}^\pm$ and $v_{\para}^\pm$ satisfy the gluing conditions
  $$
  v_{\para}^-(\rho_*, \tau) > v_{\out}^-(\rho_*,\tau),
  \qquad v_{\out}^-(3\rho_*, \tau) > v_{\para}^-(3\rho_*, \tau);
  $$
  $$v_{\para}^+(\rho_*, \tau) < v_{\out}^+(\rho_*, \tau), \qquad v_{\out}^+(3\rho_*, \tau) < v_{\para}^+(3\rho_*, \tau),$$
  for all $\tau < \tau_*$.
\end{theorem}

\begin{proof}
  Write $\bar \epsilon, \bar \delta, \bar \gamma$ for $(1 \pm \epsilon), (1 \pm \delta), (1 \pm \gamma_\pm)$.

  Write both $v_{\para}^\pm$ and $v_{\out}^\pm$ in the parabolic coordinates:
  $$
  v_{\para}^{\pm} =
  \bar \gamma \rho^{2b}(1 + \alpha \rho^{-2})^{1+b}e^{b\tau}
  \pm \frac{D}{\rho^4}e^{2b\tau}
  $$
  \begin{align*}
    v_{\out}^{\pm} &=
    \bar \delta r^{2b} (1 + t r^{-2} (\bar \epsilon (b+1) a + O(r^{2b}))),
    \qquad r \to 0 \\
    &= \bar \delta \rho^{2b}(1 + \bar \epsilon(b+1) a \rho^{-2})e^{b\tau} + O(e^{2b\tau})
    \qquad \tau \to -\infty
  \end{align*}
  This makes it clear that
  \begin{align*}
    \lim_{\tau \to -\infty}\frac{v_{\out}^\pm(\rho, \tau)}{v_{\para}^\pm(\rho, \tau)}
    &=
    \frac{\bar \delta}{\bar \gamma}
    \frac{1 + \bar \epsilon (1 + b) a \rho^{-2}}{(1 + a \rho^{-2})^{1+b}} \\
    &\eqdef
    \frac{\bar \delta}{\bar \gamma}
    H(\rho)
  \end{align*}
  We choose $\bar \gamma = \bar \gamma(\epsilon, \delta, \rho_*)$ to be
  $$\bar \gamma = \bar \delta H(2 \rho_*),$$
  which, recalling $\bar \gamma = (1 \pm \gamma)$ and $\bar \delta = (1 \pm \delta)$, means
  $$\gamma_\pm = \pm ((1 \pm \delta)H(2\rho_*) - 1)$$
  and therefore
  $$ \lim_{\tau \to -\infty}\frac{v_{\out}^\pm(2\rho_*, \tau)}{v^\pm_{\para}(2\rho_*, \tau)} = 1,$$
  as foretold in the outline before the proof.

  However, we must show that both $\gamma_+$ and $\gamma_-$ are positive.  For this, it suffices to show that $H(2\rho_*)<1$ in the case of a subsolution, and $H(2\rho_*)>1$ in the case of a supersolution.  Figure \ref{figure:Hrho} has the graph of $H(\rho)$.  Note that
  $$\lim_{\rho \to \infty}H(\rho) = 1.$$
  Computing the derivative of $H(\rho)$ shows that $H$ increasing is equivalent to
  $$a (1 \pm \epsilon) \rho^{-2} \geq \pm \epsilon$$
  In the case of the subsolution this shows that $H(\rho)$ is strictly increasing, so indeed $H(2\rho_*)<1$.  In the case of a supersolution, $H(\rho)$ is strictly decreasing as long as $\rho > \frac{\sqrt{ab(\epsilon + 1)}}{\sqrt{\epsilon}}$.  Thus by possibly increasing the constant $C$ in the definition of $\rho_*$ in Theorem \ref{outerbarriertheorem}, we can ensure that $H(\rho)$ is strictly decreasing on $[\rho_*, \infty)$.  Thus $H(2\rho_*)>1$ in the case of a supersolution.  These considerations show that both $\gamma_+$ and $\gamma_-$ are positive.

\begin{figure}
  \centering
  \includegraphics[scale=1]{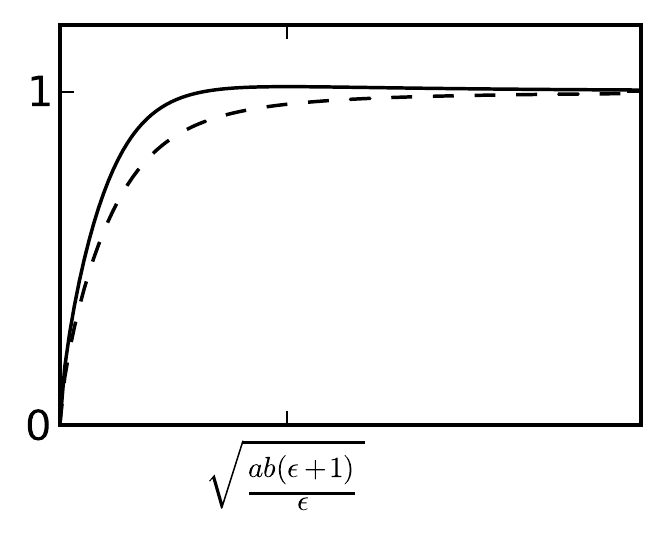}
  \caption{$H(\rho)$ for the case of a supersolution (solid) and subsolution (dashed).}
  \label{figure:Hrho}
\end{figure}
  
  Finally, because $H(2\rho_*)$ is increasing on $[\rho_*, 3\rho_*]$ in the case of a subsolution, and decreasing on $[\rho_*, 3\rho_*]$ in the case of a supersolution, the inequalities in the statement of the theorem hold at $\tau = -\infty$, and we can choose $\tau_*$ small enough so that they hold for $\tau < \tau_*$.
\end{proof}

\subsection{Gluing the Parabolic and Inner Barriers}
\label{sec:para_inner_glue}
In this section, we show the gluing conditions for the parabolic and inner barriers.  The proof is similar to the one for the gluing conditions for the outer and parabolic barriers.  Instead of considering the quotient of the barriers, we consider the difference.

\begin{theorem}\label{inoutglue}
  For any $\rho_*, \tau_*, \gamma_\pm$, we can choose $D, \sigma_* > 0$,  $\kappa_{\pm} \in [\kappa_0/2, 2\kappa_0]$, and find $t_*$ small so that
  $v_\para^\pm$ and $v_\inner^\pm$ satisfy the gluing conditions
   $$
  v_{\inner}^-(\sigma_*, t) > v_{\para}^-(\sigma_*,t),
  \qquad v_{\para}^-(3\sigma_*, t) > v_{\inner}^-(3\sigma_*, t);
  $$
  $$v_{\inner}^+(\sigma_*, t) < v_{\para}^+(\sigma_*, t), \qquad v_{\para}^+(3\sigma_*, t) < v_{\inner}^+(3\sigma_*, t),$$
  for $t < t_*$.
\end{theorem}

\begin{proof}
  The limit of $v_\inner^\pm$, as $t \to 0$ (for fixed $\sigma$) is
  $$\bry(\kappa_{\pm}\sigma) = (\kappa_{\pm}\sigma)^{-2} + O((\kappa_{\pm}\sigma)^{-4}).$$
  Since we're demanding that $\kappa_{\pm}$ lie in the closed interval $[\frac{\kappa_0}{2}, 2\kappa_0]$, this means that we can write
  $$\lim_{t \to 0}v_\inner^\pm = (\kappa_{\pm} \sigma)^{-2} + C h(\sigma)\sigma^{-4}$$
  where $|h(\sigma)| < 1$

  Written in terms of  $(\sigma, t)$, the barriers $v_{\para}^\pm$ are
  $$v_{\para}^\pm = \sigma^{-2}(\sigma^2t^{-b} + a)^{1+b} \pm D \sigma^{-4}$$
  which, in the limit $t \to 0$ are
  $$(\kappa_0\sigma)^{-2} \pm D \sigma^{-4}$$
  where we recall that $\kappa_0 = a^{-(b+1)/2}$.
  
  So we have
  \begin{equation}
    \label{glueInnerLimit}
    \tilde H(\sigma) \defeq \lim_{t \to 0}\sigma^2(v_\para^\pm - v_\inner^\pm) = \kappa_0^{-2} - \kappa_{\pm}^{-2} \pm (D\mp Ch(\sigma))\sigma^{-2}.
  \end{equation}
  
  Now, force $D > 2C$ so that $\frac{D}{2} < (D \mp Ch(\sigma)) \leq \frac{3D}{2}$.  For the rest of the proof we take the case of a supersolution.  Calculate
  $$
  \tilde H(\sigma_*) \geq \kappa_0^2 - \kappa_+^2 + \frac{D}{2}\sigma_*^{-2}
  $$
  $$
  \tilde H(3\sigma_*) \leq \kappa_0^2 - \kappa_+^2 + \left(\frac{3D}{2}\right)(3\sigma_*)^{-2}
  = \kappa_0^2 - \kappa_+^2 + \frac{D}{6}\sigma_*^{-2}
  $$
  So we see we can choose $\kappa_+ = \kappa_+(D, \sigma_*) > \kappa_0$ so that $H(3\sigma_*) < 0 < H(\sigma_*)$: choose it so that $\kappa_0^2-\kappa_+^2 = -\frac{D}{3}\sigma_*^{-2}$.  However, we must first choose $\sigma_*$ large enough (depending on $D$) so that the resulting $\kappa_+$ lands in the interval $[\kappa_0, 2\kappa_0]$.

  Finally choose $t_*$ small enough so that the inequalities hold for $t < t_*$.

\end{proof}
  
\section{Bounds Yielding Convergence of Regularized Solutions}
\label{sec:conv}

In this section we find regularizations of the initial metric and show that they converge to a solution to Ricci flow coming out of our singular initial data.  In section \ref{section:maxprinciples} we give the maximum principle we will use to show that the regularizations stay within our barriers.  We also construct a special barrier, which is used to prove the part of Theorem \ref{maintheorem} which claims that the solution must stay compact.  In section \ref{section:curvaturebnds} we find bounds which any solution within the barriers must satisfy. In section \ref{section:regularizations} we construct modifications of the initial metric, and in section \ref{section:convergencefinal} we explain why and how the regularizations converge.

\subsection{Maximum Principles}\label{section:maxprinciples}
In \cite{ACK}, the authors state and prove the following maximum principle.
\begin{lemma}\label{maxprinc}
  Let $v^-$ and $v^+$ be nonnegative, sub- and super-solutions of $v_t = \fop[v]$, that is:
  $$\left( \pars{}{t} - \fop \right) v^+ \geq 0 \qquad \left( \pars{}{t} - \fop \right) v^- \leq 0$$
  Here $v^-$ may be the maximum of smooth subsolutions, and $v^+$ may be the minimum of smooth supersolutions, which satisfy the gluing conditions.

  Assume either $v^-$ or $v^+$ satisfies
  $$v_{rr} \leq C, \qquad -\frac{v_r}{2r} \leq C, \qquad \frac{1-v}{r^2} \leq C$$
  for some $C<\infty$, on $\Xi = [0, \bar r] \times [0, \bar t]$.

  If $v^-(\bar r, t) \leq v^+(\bar r, t)$ and $v^-(0, t) \leq v^+(0, t)$ hold for $0 \leq t \leq \bar t$, and $v^-(r, 0) \leq v^+(r, 0)$ holds for $0 \leq r \leq \bar r$,  then $v^- \leq v^+$ holds throughout $\Xi$.
\end{lemma}

We need a way to verify the inequality on the right edge of the domain $r = \bar r$ in the hypothesis of the theorem above.  To do this, we will use the barriers in the following lemma from \cite{ACK}.  We call these barriers the collars, and they give us a way to connect what happens close to the North Pole with what happens in the rest of the manifold.

\begin{lemma}\label{collarlemma}
  For any $0 < m < 1$, $\bar r >0$, and small enough $\alpha > 0$, there exists $T = T(\alpha) >0$ such that for $0 < t < T(\alpha)$, the functions
  $$v_{\text{col}}^-(r, t) =
  \max
  \left\{
  0, m_-e^{-t/\alpha^2} - \left( \frac{r-\bar r}{\alpha} \right) ^2
  \right\}
  $$
  $$v_{\text{col}}^+(r, t) =
  \min
  \left\{
  1, m_+e^{3t/\alpha^2} + \left( \frac{r-\bar r}{\alpha} \right) ^2
  \right\}
  $$
  are sub- and super-solutions of $v_t = \fop[v]$, respectively.
\end{lemma}

The following lemma is the analogue to Lemma 2 of \cite{ACK}.  It allows us to conclude that there are no noncompact solutions coming out of the singularity, as in Theorem 2 of \cite{ACK}.  The main reason is that near the North Pole, the arclength between $r = r_1$ and $r = r_2$ is given by $\int_{r_1}^{r_2}v(r)^{-1/2}dr$, which is integrable at $r=0$ if we have the lower bound $v \geq r^{2b_*}$. Notice the hypotheses do not impose boundary conditions at $r=0$, only that the equation is satisfied in the interior, so we can apply it to a solution with $v(r, 0) = v_{\init}$.
\begin{lemma}\label{cptnesssubsoln}
  Let $b_* \in (1/2, 1)$.  Let $v$ be a nonnegative smooth function satisfying 
  $$\left(\pars{}{t} - \fop\right)v \geq 0$$
  on $(0, r_0] \times [0, t_0]$ and
  $$v(r, 0) > r^{2b_*}$$
  for $r \in (0, r_0]$.  Then $v \geq r^{2b_*}$ on $(0, r_1) \times [0, t_1]$ for some small $r_1, t_1$.
\end{lemma}

\begin{proof}
  Let $\bar v(r) = r^{2b_*}$.  Choose $r_1 = \left( \frac{ab_*}{2b_*^2+b_*+a} \right) ^{1/4b_*}$, or $r_1 = r_0$ if that is smaller.  Demand that $t_1$ is small enough so that $\bar v(r_1) < v(r_1, t)$ for $0 \leq t \leq t_1$.
  For arbitrary $\epsilon > 0$ let
  $$f_\epsilon(r, t) = v(r, t) - \bar v (r)  + \epsilon (1+t).$$
  We prove $f_\epsilon(r, t) > 0$ for $0 < r < r_1$ and $0 < t < t_1$.  The lemma follows by sending $\epsilon $ to 0.

  From our demand on $t_1$ and our hypotheses, we have
  $$f_\epsilon(r, 0) \geq \epsilon, \qquad \text{ for all } r \in [0, r_1]$$
  and 
  $$f_\epsilon(0, t) \geq \epsilon, \qquad f_\epsilon(r_1, t) \geq \epsilon,  \qquad \text{ for all } t \in [0, t_1].$$
  Suppose for contradiction we have a first time $\bar t \in (0, t_1]$ where $f_\epsilon$ has a zero, and take $\bar r$ minimal such that $f(\bar r, \bar t) = 0$.

  Then, on one hand, since we have $(\pars{}{t} - \fop)v\geq 0$,
  \begin{equation}\label{cptnessineq:fopneg}
    0 \geq \pars{f_{\epsilon}}{t}(\bar r, \bar t) \geq \fop[v(\bar r, \bar t)] + \epsilon
  \end{equation}

  On the other hand, we have
  $$0 = \pars{f_{\epsilon}}{r}(\bar r, \bar t), \qquad 0 \leq \pars{^2f_{\epsilon}}{r^2}(\bar r, \bar t)$$
  so at the point $(\bar r, \bar t)$
  $$v < \bar v , \qquad v_r = \bar v_r, \qquad v_{rr} \geq \bar v_{rr}$$
  so using the formula (\ref{riccievo}) for $\fop[v]$, 
  \begin{align*}
    \bar r^2 \fop[v(\bar r, \bar t) ]
    &=
    \frac{a}{2}\bar rv_r + av
    + \bar r^2 v_{rr}v - \frac{1}{2}(\bar r v_r)^2
    - \bar r v_r v - av^2 \\
    &\geq
    \frac{a}{2}\bar r \bar v_r 
    - \frac{1}{2}(\bar r \bar v_r)^2
    - \bar r \bar v_r \bar v - a\bar v^2 \\
    &=
    abr^{2b}
    - (2b^2 + b + a)r^{4b}
  \end{align*}
  In the first inequality we used that $v$ and $v_{rr}v \geq \bar v_{rr} v$ are positive.  From the last line, we see from our choice of $r_1$ above that $\bar r \fop[v(\bar r, \bar t)] > 0$, contradicting (\ref{cptnessineq:fopneg}).
\end{proof}

\subsection{Curvature Bounds for Solutions within the Barriers}\label{section:curvaturebnds}
Lemmas \ref{curvaturebndfixedtime} and \ref{curvaturebndfixedr} let us understand the curvature of anything between the barriers created in Section \ref{sec:barriers}.  Lemma \ref{curvaturebndfixedtime} gives, in particular, an upper bound on curvature strictly away from $t=0$, and Lemma \ref{curvaturebndfixedr} gives an upper bound on curvature strictly away from $r=0$.   Recall that $K = \frac{-v_r}{2r}$ and $L=\frac{1-v}{r^2}$ are the sectional curvatures of the manifold.  The results of this section are entirely on the level of the PDE $v_t = \fop[v]$ on $(0, r_*) \times (0, t_*)$.

\begin{lemma}\label{curvaturebndfixedtime}
  Let $v(r, t)$ be a solution to $v_t = \fop[v]$ on $(0, r_*) \times (0, t_*)$.  Let $v^\pm$ be the barriers created in Section \ref{sec:barriers}.  Suppose that $v$ satisfies $v^- < v < v^+$, as well as the assumptions \ref{vbnd} and \ref{abnd} on the initial data, i.e. that
  $$|v(r, t)| < 1, \qquad \text{for }r>0,$$
  $$|r^2 (L-K)| = |(1-v)+\frac{1}{2}rv_r| < A.$$
  There is $C = C(A, r_*, t_*, \delta, \epsilon) < \infty$ such that
  $$
  \frac{1}{C} \frac{1}{t^{1+b}}
  \leq 
  \sup_{r \in (0, r_*)}|K| + |L|
  \leq
  C\frac{1}{t^{1+b}}
  $$
\end{lemma}

\begin{remark}
  The exact upper bound $C \frac{1}{t^{1+b}}$ is not important for the following sections; but it is important that the bound does not depend on $v$.  We will not use the lower bound.  The purpose of the more exact estimates is to give the conclusion $|\Rm| \sim \frac{1}{t^{1+b}}$ in Theorem \ref{maintheorem}.
\end{remark}

\begin{proof}
  The fact that $v$ lies between our barriers will give us an immediate bound on $|L|$.  As we mentioned when describing the initial data, we will have to use the bound on $|r^2 (L-K)|$ to bound $K$.

  As $v < 1$,
  $$\left| \frac{1-v}{r^2} \right| = \frac{1-v}{r^2} \leq \frac{1-v^-}{r^2}$$
  Outside of the inner region  $\sigma \geq \sigma_*$ by definition, which translates to
  $\frac{1}{r^2} \leq \frac{1}{\sigma_*^2 t^{1+b}}.$
  In particular $|L| \leq C\frac{1}{t^{1+b}}$.
  In the inner region, we have 
  $$\left| \frac{1-v}{r^2} \right| \leq \frac{1-v_{\inner}^-}{r^2} = \frac{1-\bry (k_-\sigma) + (1-\epsilon) k^{-2}_- \cry(k_-\sigma)\theta \theta_t}{r^2}.$$
  which, checking the asymptotics of $\bry$ and $\cry$, gives the bound
  $$|L| = \left| \frac{1-v}{r^2} \right| \leq C \frac{\sigma^2 }{r^2} = C \frac{1}{t^{1+b}}.$$
  Thus, the bound $|L| \leq C\frac{1}{t^{1+b}}$ holds throughout.  The same argument implies the lower bound in the statement of the theorem:
  $$|K|+|L| \geq |L| = \left| \frac{1-v}{r^2} \right| \geq \frac{1-v_{\inner}^+}{r^2} \geq \frac{1}{C}\frac{1}{t^{1+b}}$$

  All that is left is the upper bound on $|K|$.  We apply the bound on $|L-K|$ from the hypotheses:
  $$|K| \leq |L| + \frac{A}{r^2} \leq C \left( \frac{1}{t^{1+b}} + \frac{1}{r^2} \right)$$

  Outside of the inner region, $\frac{1}{r^2}$ is dominated by $\frac{1}{t^{1+b}}$.  Therefore the bound $|K| \leq C \frac{1}{t^{1+b}}$ follows for $\frac{r^2}{t^{1+b}}<\sigma_*^2$.
  
  In the inner region, the rescaling argument from \cite[Lemma~10]{ACK} carries through.  For arbitrary $(\bar r, \bar t)$ in the inner region, consider
  $$w(x, y)
  = \frac{\bar t^{1+b}}{\bar r^2}(1 - v(\bar r + \bar r x, \bar r + \bar r^2 y))
  = \bar \sigma^{-2} (1 - v(\bar r + \bar r x, \bar r + \bar r^2 y)).
  $$
  One checks that $w$ is bounded and solves a strictly parabolic equation on $[-\frac{1}{2}, \frac{1}{2}] \times [-\frac{1}{2}, \frac{1}{2}]$.  (The coefficient of $w_{xx}$ is $(1-\bar \sigma^2w) = v$, which is bounded from below in the inner region.)  Furthermore, $w$ and the ellipticity constant are bounded independently of $(\bar r, \bar t)$.  Applying interior estimates yields $w_x(0,0) < C$ for some constant independent of $(\bar r, \bar t)$, and scaling back we find $v_r(\bar r, \bar t) < C\frac{1}{\bar t^{1+b}}$.
\end{proof}

\begin{lemma}\label{curvaturebndfixedr}
  Let $v(r, t)$ be as in the hypotheses of Lemma \ref{curvaturebndfixedtime}.  Then for any $r_1$ there is $C = C(r_1, r_*, t_*, \delta, \epsilon)$ such that for $r \in (r_1, r_*)$
  $$\sup_{(r, t) \in (r_1, \frac{r_1+r_*}{2}) \times (0, t_*)}   |K| + |L| \leq C.$$
\end{lemma}

\begin{proof}
  On $(r_1/2, r_*) \times (0, t_*)$, $v$ satisfies $v_t = \fop[v]$.  As $v > v^- > 0$ this is a parabolic equation, with the ellipticity constant only dependent on $r_1$, $r_*$, $t_*$, $\delta$, and $\epsilon$.  The boundary data also satisfies bounds dependent only on the same numbers.  Therefore we have an upper bound on $v_r$ in the smaller region $(r_1, \frac{r_1+r_*}{2}) \times (0, t_*)$.  Looking at $K = \frac{- v_r}{2r}$ and $L = \frac{1-v}{r^2}$ we see we have proved the theorem.
\end{proof}

\begin{remark}
  In section \ref{section:regularizations} we want to apply Lemmas \ref{curvaturebndfixedtime} and \ref{curvaturebndfixedr}, but instead of the hypothesis $v^- < v < v^+$ we will have
  $$v^-(r, t + \omega) < v(r, t) < v^+(r, t+\omega).$$
  If we define
  $$ \tilde v^-(r, t) = \sup_{\omega \in [0, \omega_{\max}]}v^-(r, t+\omega)$$
  $$ \tilde v^+(r, t) = \inf_{\omega \in [0, \omega_{\max}]}v^+(r, t+\omega)$$
  then for small enough $\omega_{\max}$, and possibly decreasing $t_*$, the proofs carry through with $v^- < v < v^+$ replaced with $\tilde v^- < v < \tilde v^+$.

\end{remark}

\subsection{Regularizations}\label{section:regularizations}

We construct the smooth evolution from our singular initial data as the limit of a sequence of regularized metrics.  We will describe the construction of the regularized metrics here.

Let $\Nb_{r_0}$ be the connected neighboorhood of the North Pole where $\psi_{\init}(x)<r_0$.  Recall $r_\#$ is defined so that $\psi_s > 0$ if $0 < \psi < 2r_\#$.  For small $\omega \geq 0$, we will construct a smooth rotationally-invariant metric $g_\omega = (ds)^2 + \psi_\omega(s)^2 g_{S^n}$ on $S^{n+1}$.   On $S^{n+1}\setminus\Nb_{\rho_* \sqrt{\omega}}$ we let $g_\omega$ coincide with our initial metric $g_{\init}$.  On $\Nb_{\rho_* \sqrt{\omega}}$ we will still have $\psi_s>0$, so we may use $r$ as a coordinate.  There, $v_\omega = \psi_{\omega,s}^2$ satisfies
$$v^-(r, \omega) < v_\omega(r) < v^+(r, \omega)$$
as well as $|r^2(K-L)| < A$.  (Recall this last condition also holds for our initial data $v_{\init}$, and it is preserved under Ricci flow.)

Lemmas \ref{mod:shorttime} through \ref{curvaturebnds} give properties of $g_\omega$ which are independent of $\omega$.  

\begin{lemma}[Uniform short time existence]
  \label{mod:shorttime}
  There is $t_* > 0$ such that the modified initial metrics $g_\omega$ all have a Ricci flow on the time interval $[0, t_*]$.  The time derivatives $\pars{g_\omega}{t}$ are bounded independently of $\omega$.  On $\Nbr{r_\#}$, $\psi_\omega$ is increasing with respect arclength from the North Pole.
\end{lemma}
Lemma \ref{mod:shorttime} is proven in Section 2.1 of \cite{ACK}.

\begin{lemma}[Barrier trapping]
  \label{mod:barrier}
  There is a $t_* > 0$ such that
  $$
  v_{\epsilon, \delta}^-(r, t+\omega)
  < v_\omega(r, t)
  < v_{\epsilon, \delta}^+(r, t+\omega)
  $$
  for all $(r,t) \in [0, r_*] \times [0, t_*]$.  
\end{lemma}

\begin{proof}
Figure \ref{figure:collar} illustrates this argument, which follows Lemma 8 of \cite{ACK}.  Choose the constants $m_{\pm}$ and $\alpha$ in the definitions of $v_{\col}$ from Lemma \ref{collarlemma} so that
  \begin{equation}\label{initcollartrap}
    v^-_{\col}(r, 0) < v_{\init}(r, 0) < v^+_{\col}(r, 0)
  \end{equation}
  for all $r \in (0, r_\#]$, and 
  \begin{equation}\label{rightedgepinch}
  v^-_{\epsilon, \delta} (r_*, 0) < v^-_{\col}(r_*, 0),
  \qquad
  v^+_{\col}(r_*, 0) < v^+_{\epsilon, \delta}(r_*, 0).
  \end{equation}
  We want these inequalities to hold for the $v_\omega$ and the time-shifted barriers.  The first inequality (\ref{initcollartrap}) actually holds for $v_{\init}$ replaced with $v_{\omega}$ (for $\omega$ small enough) since we do not modify $v_{\init}$ near $r_*$, and near 0
  $$v^-_{\col}(r, 0) = 0 < v_{\omega}(r, 0) < 1 = v^+_{\col}(r,0).$$
  Shrink $t_*$ and demand $\omega$ is small enough to turn the second inequality (\ref{rightedgepinch}) into
  $$
  v^-_{\epsilon, \delta}(r_*, t + \omega) < v^-_{\col}(r_*, t),
  \qquad
  v^+_{\col}(r_*, t) < v^+_{\epsilon, \delta}(r_*, t + \omega).
  $$
  for $0 < t < t_0$.
  
  Now we apply the maximum principle (Lemma \ref{maxprinc}) twice.  First, we apply it with the collar sub- and super-solutions. For all $\omega$ we have $|\psi_s| < 1$, and this is preserved under Ricci flow.  Since $0 < v(r, t) < 1$ the condition on the right edge of the domain is satisfied.  Therefore $v_\omega$ stays between the collars, in particular
  $$
  v^-_{\epsilon, \delta} (r_*, t + \omega) < v^-_{\col}(r_*, t)
  < v_{\omega}(r_*, t)
  < v^+_{\col}(r_*, t) < v^+_{\epsilon, \delta}(r_*, t + \omega)
  $$
  which allows us to apply the maximum principle to $v^\pm_{\epsilon, \delta}$, on $[0, r_*] \times [0, t_0]$.
  \begin{figure}[t]
    \includegraphics[scale=.9]{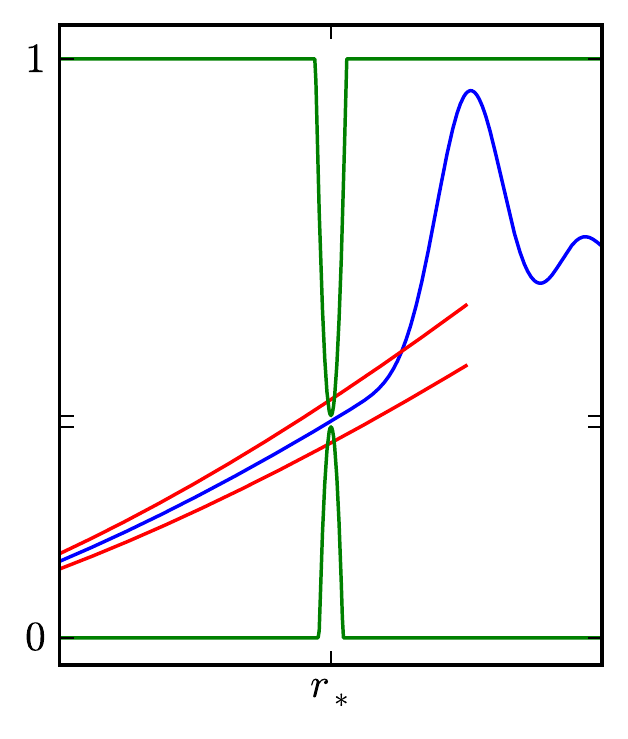}
    \includegraphics[scale=.9]{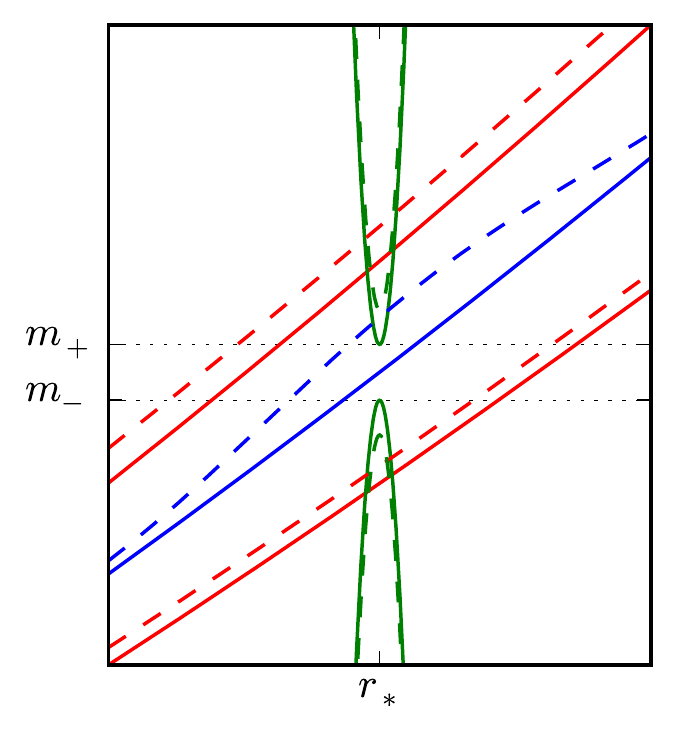}
    \caption{Left: The collars (green), separating $v^\pm(r_*, 0)$ (red) from $v_{\init}(r_*)$ (blue).  Right: A closer look at the initial situation (solid) and the situation after a small time (dashed). The collars are still keeping the solution away from $v^\pm(r_*, t)$.}
    \label{figure:collar}
  \end{figure}
\end{proof}

\begin{lemma}[Bounds away from the singularity]\label{curvaturebnds}
  For fixed $t_1 > 0$ and any $p \in \nats$ (possibly 0) there is $C_p = C_p(p,  \epsilon, \delta)$ such that the regularized solutions satisfy
  $$|\nabla^p \Rm| < \frac{C_p}{t_1^{1+b}}$$
  on $S^{n+1} \times [t_1, t_*]$.

  For fixed $r_1 > 0$ and any $p \in \nats$ (possibly 0) there is $C_p' = C_p'(r_1, p, \epsilon, \delta)$ such that the regularized solutions satisfy
  $$|\nabla^p \Rm| < C_p'$$
  on $(S^{n+1}\setminus\Nb_{r_1}) \times [0, t_*]$.
\end{lemma}

\begin{remark}
  The $t^{-(1+b)}$ factor in the bound for positive time is not important, except for the claim in our main theorem that $|\Rm| \sim t^{-(1+b)}$.
\end{remark}

\begin{proof}
  The full proof is Lemma 11 of \cite{ACK}.  Here we only outline where the estimates come from.

  The estimates for $|\Rm|$ come in the following ways:
  \begin{itemize}
  \item The bound on $|\Rm|$ in $\Nb_{r_*} \times [t_1, t_*]$ comes directly from Lemma \ref{curvaturebndfixedtime}.
  \item The bound on $|\Rm|$ in $(\Nb_{r_*} \setminus \Nb_{r_1}) \times [0, t_*]$ comes directly from Lemma \ref{curvaturebndfixedr}.
  \item The bound on $|\Rm|$ in $(S^{n+1} \setminus \Nb_{r_*})$ comes from the parabolic maximum principle applied to
    $$\pars{|\Rm|^2}{t}  \leq \Delta |Rm|^2 + c_n |Rm|^3$$
  \end{itemize}
  Then the estimates for the derivatives of $\Rm$ come in the following ways:
  \begin{itemize}
  \item The bounds in $(\Nb_{r_*} \setminus \Nb_{r_1}) \times [0, t_*]$ come directly from Lemma \ref{curvaturebndfixedr}.
  \item The bounds in $(S^{n+1}\setminus\Nb_{r_*}) \times [0, t_*]$ come from modifications of Shi's local estimates (\cite[Appendix]{ModifiedShi} or \cite[14.4.1]{bookanalytic}).  The modifications allow us to use the fact that we know $g_{\omega}(0) = g_{\init}$ to get bounds on the derivatives of $|\Rm|$ which are uniform in time.
  \item The bounds in $S^{n+1} \times [t_1, t_*]$ come from Shi's global estimates (\cite[Theorem~1.1]{Shi} applied to the metric at $t_1/2$.
  \end{itemize}
\end{proof}

\subsection{Convergence to a Solution}\label{section:convergencefinal}

We expect to prove convergence of the solutions $g_\omega$, on any compact subset of $S^{n+1} \times [0,t_*]$ which does not contain $(\NP, 0)$.

At the present moment, however, showing convergence of the metrics $g_\omega$ near the  point $\NP$ is difficult because we have only defined them up to some diffeomorphisms.  The regularized metric $g_\omega(0)$ coincides with $g_{\init}$ on $S^{n+1} \setminus \Nb_{\rho_*\sqrt{\omega}}$ but in $\Nb_{\rho_*\sqrt{\omega}}$ they are only defined in terms of $v_\omega$.

A metric on $[0, x_{\#}] \times S^n$ of the form
$$g(x, p) = \varphi(x)^2dx^2 + \psi(x)^2g_{S^n}$$
with $\psi'(x)>0$ and $ \psi(0) = 0$ is determined by $v(r) = \frac{\psi'(\psi^{-1}(r))}{\varphi(\psi^{-1}(r))}$, but only up to pullbacks by diffeomorphisms which are invariant under the $SO(n+1)$ action.  Such a diffeomorphism given by
$$\Phi(x, p) = (f(x), p)$$
yields the pulled-back metric
$$
(\Phi^*g)(y, p) = \frac{\varphi(f^{-1}(y))^2}{f'(f^{-1}(y))^2} (dy)^2 + \psi(f^{-1}(y))g_{S^n}
$$

From this we see that by applying such a diffeomorphism we can destroy any bound of the form $c g \leq g_{\omega} \leq Cg$ or $|\nabla_g^pg_{\omega}| \leq C_p$ on $[t_1, t_*] \times S^{n+1}$.  We would like to deal with this problem by using Hamilton's compactness result \cite{H}.  We could apply this to $S^{n+1} \times [t_1, t_*]$ and use a diagonal argument to take $t_1$ to 0.  This would give us a sequence of diffeomorphisms $\Phi_k:S^{n+1} \to S^{n+1}$ so that the pullbacks $\Phi_k^*g_{\omega_k}$ converge in $C^\infty_{\text{loc}}(S^{n+1} \times (0, t_*])$.  Then we would want to show that away from the North Pole, $\Phi_k^*g_{\omega_k}$ converges as $t \searrow 0$ to some pullback of $g_{\init}$.  The issue is that we would have no control on $\Phi_k$.

Instead, we first show that there is a subsequence $g_{\omega_k}$ which converge to $g_{*}$ in $C^{\infty}_{\loc}((S^{n+1}\setminus\NP) \times [0, t_*])$.  This is easy because for any $t_1$ the metrics eventually agree with $g_{\init}$ on $[t_1, t_*]$.  Then we find a sequence of diffeomorphisms $\Phi_k$ so that $\Phi_k^*g_k$ converge in $C^{\infty}_{\loc}(S^{n+1} \times (0, t_*])$, and so that $\Phi_k$ converge to a diffeomorphism $\Phi$.  Then $\Phi_k^*g_k$ converge to $\Phi^*g_{*}$ in $C^{\infty}_{\loc}(S^{n+1} \times [0, t_*] \setminus (\NP, 0))$.  Notice that by then applying $(\Phi^{-1})^*$ we get that $g_k$ converge to $g_*$ in $C^{\infty}_{\loc}(S^{n+1} \times [0, t_*] \setminus (\NP, 0))$.

Section 6.4 of \cite{ACK} accomplishes the same thing as we do here, but instead finds the solution away from the North Pole, then shows that it is smooth by controling the evolution forwards and backwards in time from some $t_1$.

The compactness result we use is Theorem \ref{cptnessthm} below.
\begin{theorem}\label{cptnessthm}
  Let $K$ be a compact subset of a manifold $M$ with background metric $g$.  Let $g_k$ be a collection of solutions to Ricci flow on a neighboorhood of $K \subset M$, on the time interval $[0, t_*]$, and let $t_0 \in [0, t_*]$.  If on $K$
  \begin{enumerate}
  \item There are $0 < c < C < \infty$ such that for all $k$
    $$ cg  \leq g_k(t_0) \leq Cg$$
  \item There are constants $C_p$, $p \geq 1$, such that for all $k$
    $$| \nabla^p_g g_k(t_0) |_g \leq C_p$$
  \item There are constants $C_p'$, $p \geq 0$ such that for all $k$ and $t \in [0, t_*]$,
    $$| \nabla^p_{g_k(t)} \Rm_{g_k(t)} |_{g_k(t)} \leq C_p'$$
  \end{enumerate}
  Then the metrics $g_k$ converge to some metric $g_*$ in $C^\infty(K \times [0, t_*
])$.
\end{theorem}
The proof of this theorem is contained in the proof of Hamilton's full compactness theorem; it is applied once metrics on a limiting manifold are constructed.  See the end of Section 2 in \cite{H} or Section 2.2.2 of Chapter 1 of \cite{bookgeometric}

\begin{theorem}
A subsequence $g_{\omega_j}$ of the regularized solutions converge to a solution $g_{*}$ of Ricci flow on $C^\infty_{\loc}(S^{n+1} \times [0, t_*] \setminus (\NP, 0))$, with $g_*(0) = g_{\init}$.
\end{theorem}

\begin{proof}
  For any $r_1$, for small enough $\omega$, the regularizations $g_\omega(0)$ coincide with $g_{\init}$ on $S^{n+1}\setminus\Nb_{r_1}$.  By Theorem \ref{cptnessthm} with the bounds from Lemma \ref{curvaturebnds}, a subsequence $\omega_k \searrow 0$ of the regularized flows converge to some metric $g_*$.  Let us call the subsequence $g_k \defeq g_{\omega_k}$, and label any geometric quantities related to $g_k$ with a $k$. By diagonalization we can extend the convergence to $C^{\infty}_{\loc}((S^{n+1}\setminus{\NP}) \times [0, t_*])$.    The limiting metric $g_*$ will be of the form
  $$g_*(x, p) = \varphi_*(x)^2(dx)^2 + \psi_*^2(x)g_{S^n}$$

  Now we describe the diffeomorphisms $\Phi_k:S^{n+1} \to S^{n+1}$ described leading up to the theorem.  We will have $\Phi_k(x, p) = (f_k(x), p)$.  Pick $t_1 \in (0, t_*)$.    Let $x_\# = \psi_{\init}^{-1}(r_{\#})$ so that all of the $\psi_k$ are increasing for $x<x_{\#}$. On $[0,x_\#/2] \times S^{n+1}$ we will have $f_k(x) = \psi_k(x, t_1)$.  We have found that $\psi_k$ converge to $\psi_*$ on $[x_\#/4, x_\#]$ (in particular), so we can extend the $f_k$ to be smooth diffeomorphisms which converge to a smooth limit $f$ on the entire interval $[0, 1]$.  Let $\Phi(x, p) = (f(x), p)$.

  The important thing is that $f_k(x) = \psi_k(x)$ near the north pole, so there
  $$(\Phi^*g_k(t_1))(y, p) = \frac{(dy)^2}{v_{k}(y, t_1)} + y^2g_{S^n}$$
  The functions $v_{\omega_k}$ are bounded from below and have uniformly bounded derivatives.  (In fact, even $K_k = -v_{k}(r,t_1)/2r$ have uniformly bounded derivatives, by Lemma \ref{curvaturebnds}.)  Thus the pullbacks $\Phi_{k}^*g_k(t_1)$ satisfy conditions 1 and 2 of Theorem \ref{cptnessthm}, near the north pole.  Away from the north pole, $\Phi_k^*g_{\omega_k}$ converges in $C^{\infty}$ to $\Phi^*g_*$, so there the conditions are satisfied as well.

  Thus $\Phi_k^*g_{\omega_k}$ converges to $\Phi^*g_*$ on $C^{\infty}_{\loc}(S^{n+1} \times [0, t_*] \setminus (\NP, 0))$.

\end{proof}
\appendix
\section{Notation}
\begin{tabular}{c c c}
  Symbol & Meaning & First use/explanation \\
  \\
  $a$ & $2(n-1)$ & Page \pageref{adef}\\
  $L, K$ & Sectional Curvatures & Section \ref{sec:basiccoordinates} \\
  $v$ & The slope function describing the metric. & Section \ref{sec:basiccoordinates} \\
  $\Nb_{r}$ & A neighborhood of the North Pole where $r < 0$. & Section \ref{sec:basiccoordinates} \\
  $b$ & Number in $(0,1)$ describing  initial data & Page \pageref{bdef} \\
  $v_{\init}$ & The slope function for the initial metric. & \\
  $v_{\out}, v_{\para}, v_{\inner}$ & Formal solutions & Section \ref{sec:formal}\\
  $\kappa_0$ & $a^{-(b+1)/2}$ & Page \pageref{k0def} \\
  $[v]_2$ & $|v| + r|v_r| + r^2|v_{r^2}|$ & Page \pageref{fnrmdef} \\
  $\epsilon, \delta, D, \gamma, \kappa_{\pm}$ & Constants appearing in the barriers & Section \ref{sec:barriers} \\
  $r_*, \rho_*, \sigma_*$ & Constants demarcating the regions & Section \ref{sec:barriers}
\end{tabular}

\bibliographystyle{halpha}
\bibliography{bibliography}

\newcommand{\etalchar}[1]{$^{#1}$}
\begin{thebibliography}{CCG{\etalchar{+}}07b}

\bibitem[ACK12]{ACK}
Sigurd~B. Angenent, Christina~M. Caputo, and Dan Knopf.
\newblock Minimally invasive surgery for {R}icci flow singularities.
\newblock {\em Journal f{\"u}r die reine und angewandte Mathematik}, Issue
  672:39--87, November 2012.

\bibitem[AIK12]{AIK}
Sigurd~B. Angenent, James Isenberg, and Dan Knopf.
\newblock Degenerate neckpinches in {R}icci flow.
\newblock 2012, arXiv:1208.4312v1 [math.DG].

\bibitem[AK04]{AK}
Sigurd~B. Angenent and Dan Knopf.
\newblock An example of neckpinching for {R}icci flow on ${S}^{n+1}$.
\newblock {\em Mathematical Research Letters}, 11(4), 2004.

\bibitem[CCG{\etalchar{+}}07a]{bookgeometric}
Bennett Chow, Sun-Chin Chu, David Glickenstein, Christine Guenther, James
  Isenberg, Tom Ivey, Dan Knopf, Peng Lu, Feng Luo, and Lei Ni.
\newblock {\em The Ricci Flow: Techniques and applications. Part I: Geometric
  Aspects}.
\newblock American Mathematical Society, 2007.

\bibitem[CCG{\etalchar{+}}07b]{bookanalytic}
Bennett Chow, Sun-Chin Chu, David Glickenstein, Christine Guenther, James
  Isenberg, Tom Ivey, Dan Knopf, Peng Lu, Feng Luo, and Lei Ni.
\newblock {\em The Ricci Flow: Techniques and applications. Part II: Analytic
  Aspects}.
\newblock American Mathematical Society, 2007.

\bibitem[Ham95]{H}
Richard~S Hamilton.
\newblock A compactness property for solutions of the {R}icci flow.
\newblock {\em American Journal of Mathematics}, 117:545--572, 1995.

\bibitem[KL14]{Singular}
Bruce Kleiner and John Lott.
\newblock Singular {R}icci flows {I}.
\newblock 2014, arXiv:1408.2271v1 [math.DG].

\bibitem[LT]{ModifiedShi}
Peng Lu and Gang Tian.
\newblock Uniqueness of standard solutions in the work of {P}erelman.

\bibitem[Per02]{Perelman}
Grisha Perelman.
\newblock The entropy formula for the {R}icci flow and its geometric
  applications.
\newblock 2002, arXiv:math/0211159 [math.DG].

\bibitem[Shi89]{Shi}
Wan-Xion Shi.
\newblock Deforming the metric on complete {R}iemannian manifolds.
\newblock {\em Journal of Differential Geometry}, 30:223--301, 1989.

\bibitem[Wu14]{Needles}
Haotian Wu.
\newblock On type-{II} singularities in {R}icci flow on ${R}^{n+1}$.
\newblock {\em Communications in Partial Differential Equations},
  39:2064--2090, 2014.

\end{thebibliography}

\end{document}